\documentclass[11pt]{article}
\usepackage{amsthm, amsmath, amssymb, amsfonts, url, booktabs, tikz, setspace, fancyhdr, amsbsy}
\usepackage[margin = 0.8in]{geometry}
\usepackage{hyperref, enumerate}
\usepackage{esint}
\usepackage{verbatim}
\usepackage{comment}
\usepackage{indentfirst}

\definecolor{refkey}{gray}{.75}
\definecolor{labelkey}{gray}{.2}

\newtheorem{theorem}{Theorem}[section]

\newtheorem{lemma}[theorem]{Lemma}

\theoremstyle{definition}

\newtheorem{remark}[theorem]{Remark}

\newcommand{\innerproduct}[2]{\langle #1, #2 \rangle}

\newcommand{\norm}[1]{\left\lVert#1\right\rVert}

\newcommand{\R}{\mathbb{R}}

\newcommand{\BB}{\boldsymbol{B}}

\newcommand{\DD}{\boldsymbol{D}}

\newcommand{\vv}{\boldsymbol{v}}
\newcommand{\SSS}{\boldsymbol{S}}

\newcommand{\ww}{\boldsymbol{w}}

\newcommand{\dx}{\,\mathrm{d}x}

\newcommand{\dt}{\,\mathrm{d}t}
\newcommand{\ds}{\,\mathrm{d}s}

\newcommand{\pa}{\partial}
\newcommand{\fff}{\boldsymbol{f}}

\newcommand{\ffpsi}{\boldsymbol{\psi}}

\newcommand{\e}{\varepsilon}

\newcommand{\calI}{\mathcal I}
\newcommand{\calJ}{\mathcal J}

\numberwithin{equation}{section}

\def\ocirc#1{\ifmmode\setbox0=\hbox{$#1$}\dimen0=\ht0
    \advance\dimen0 by1pt\rlap{\hbox to\wd0{\hss\raise\dimen0
    \hbox{\hskip.2em$\scriptscriptstyle\circ$}\hss}}#1\else
    {\accent"17 #1}\fi}

\begin{document}

\title{On the local well-posedness of strong solutions to the unsteady flows of shear-thinning non-Newtonian fluids with a concentration-dependent power-law index}

\author{
Kyueon Choi,\thanks{Department of Mathematics, Yonsei University, Seoul, Republic of Korea. Email: \tt{gyueon@yonsei.ac.kr}}
~Kyungkeun Kang\thanks{Department of Mathematics, Yonsei University, Seoul, Republic of Korea. Email: \tt{kkang@yonsei.ac.kr}}
~and~Seungchan Ko\thanks{Department of Mathematics, Inha University, Incheon, Republic of Korea. Email: \tt{scko@inha.ac.kr}}
}

\date{}

\maketitle

~\vspace{-1.5cm}

\begin{abstract}
We investigate a system of nonlinear partial differential equations modeling the unsteady flow of a shear-thinning non-Newtonian fluid with a concentration-dependent power-law index. The system consists of the generalized Navier--Stokes equations coupled with a convection-diffusion equation describing the evolution of chemical concentration. This model arises from the mathematical description of the behavior of 
synovial fluid in the cavities of articulating joints. We prove the existence of a local-in-time strong solution in a three-dimensional spatially periodic domain, assuming that $\frac{7}{5} < p^- \le p(\cdot) \le p^+ \le 2$, where $p(\cdot)$ denotes the variable power-law index and $p^-$ and $p^+$ are its lower and upper bounds, respectively. Furthermore, we prove the uniqueness of the solution under the additional condition $p^+ < \frac{28}{15}$. In particular, our three-dimensional analysis directly implies the existence and uniqueness of solutions in the two-dimensional case under the less restrictive condition $1 < p^- \le p(\cdot) \le p^+ \le 2$.
\end{abstract}

\noindent{\textbf{Keywords}{: Non-Newtonian fluid, convection-diffusion equation, synovial fluid, shear-thinning fluid, strong solution, variable power-law index}

\smallskip

\noindent{\textbf{AMS Classification:} 76D03, 35Q35, 35Q92, 76R50

\section{Introduction}
In this paper, we are interested in the existence of a strong solution describing the unsteady motion of a chemically reacting fluid within a three-dimensional periodic domain. To be more specific, we shall investigate the local well-posedness of the following system of equations:
\begin{align}
\pa_t\vv + (\vv \cdot \nabla) \vv - {\rm div}\SSS(c, \DD \vv)  &= - \nabla \pi  + \fff , \label{main_sys} \\
{\rm div}\, \vv &= 0, \\ 
\pa_t c + {\rm div}(c \vv) -\Delta c  &= 0, \label{concentration}
\end{align}
in $Q_T :=\Omega\times I$, where $\Omega=[0,1]^3$ and $I=[0,T]$ with $T>0$. Here, $\vv: Q_T \to \mathbb{R}^3$ denotes the velocity field, $\pi: Q_T \to \mathbb{R}$ is the kinematic pressure, and $c: Q_T \to \mathbb{R}_+$ represents the concentration distribution. The external force is given by $\fff: Q_T \to \mathbb{R}^3$, and the Cauchy stress tensor is denoted by $\SSS(c, \DD \vv): Q_T \to \mathbb{R}^{3 \times 3}$, which depends on the concentration and the symmetric part of the velocity gradient, defined as $\DD \vv = \frac{1}{2} \left( \nabla \vv + (\nabla \vv)^{\top} \right)$. We consider the initial-boundary value problem corresponding to equations \eqref{main_sys}-\eqref{concentration}, subject to periodic boundary conditions in the spatial domain $[0,1]^3$ for both the velocity field $\vv$ and the concentration $c$. The initial conditions are specified by
\begin{equation*} 
\vv(x,0) = \vv_0(x) \quad\text{and}\quad c(x,0) = c_0(x) \quad \text{for } x \in \Omega.
\end{equation*}
Regarding the Cauchy stress tensor $\SSS$, we consider the constitutive relation of the form
\begin{equation*}
\SSS(c, \DD \vv) = 2 \nu(c, |\DD \vv|)\, \DD \vv,
\end{equation*}
where the viscosity function $\nu$ is given by
\begin{equation}\label{visc_form}
\nu(c, |\DD \vv|) = \nu_0 \left(1 + |\DD \vv|^2\right)^{\frac{p(c)-2}{2}},
\end{equation}
with a fixed constant $\nu_0 > 0$. The exponent $p: \mathbb{R}_{\ge 0} \to \mathbb{R}_{\ge 0}$ is assumed to be a Lipschitz continuous function satisfying
\[
1 < p^- \leq p(c) \leq p^+ < \infty,
\]
for some positive constants $p^-$ and $p^+$. Then it is straightforward to verify the following properties of the Cauchy stress tensor. There exist positive constants $K_1$, $K_2$ and $K_3$ such that for all $(c, \DD) \in \mathbb{R} \times \mathbb{R}^{3 \times 3}$ and for any $\BB \in \mathbb{R}^{3 \times 3}$, the following inequalities hold:
\begin{equation} \label{P1}
\begin{split}
\frac{\partial \SSS(c, \DD)}{\partial \DD} : (\BB \otimes \BB) &\geq K_1 \left(1 + |\DD|^2 \right)^{\frac{p(c) - 2}{2}} |\BB|^2, \\
\left| \frac{\partial \SSS(c, \DD)}{\partial \DD} \right| &\leq K_2 \left(1 + |\DD|^2 \right)^{\frac{p(c) - 2}{2}}, \\
\left| \frac{\partial \SSS(c, \DD)}{\partial c} \right| &\leq K_3 \left(1 + |\DD|^2 \right)^{\frac{p(c) - 1}{2}} \log\left(2 + |\DD|\right),
\end{split}
\end{equation}
where the tensor product $(\BB \otimes \BB)$ is defined componentwise by $(\BB \otimes \BB)_{ijkh} = B_{ij} B_{kh}$. Moreover, these properties imply the following monotonicity condition: there exists a constant $K_4 > 0$ such that for any $c \in \mathbb{R}$ and any $\DD_1, \DD_2 \in \mathbb{R}^{3 \times 3}$, there holds
\begin{equation} \label{P2}
\left( \SSS(c, \DD_1) - \SSS(c, \DD_2) \right) : (\DD_1 - \DD_2) \geq K_4 \left(1 + |\DD_1|^2 + |\DD_2|^2 \right)^{\frac{p(c) - 2}{2}} |\DD_1 - \DD_2|^2.
\end{equation}

\begin{figure}
    \centering    \includegraphics[width=0.4\textwidth]{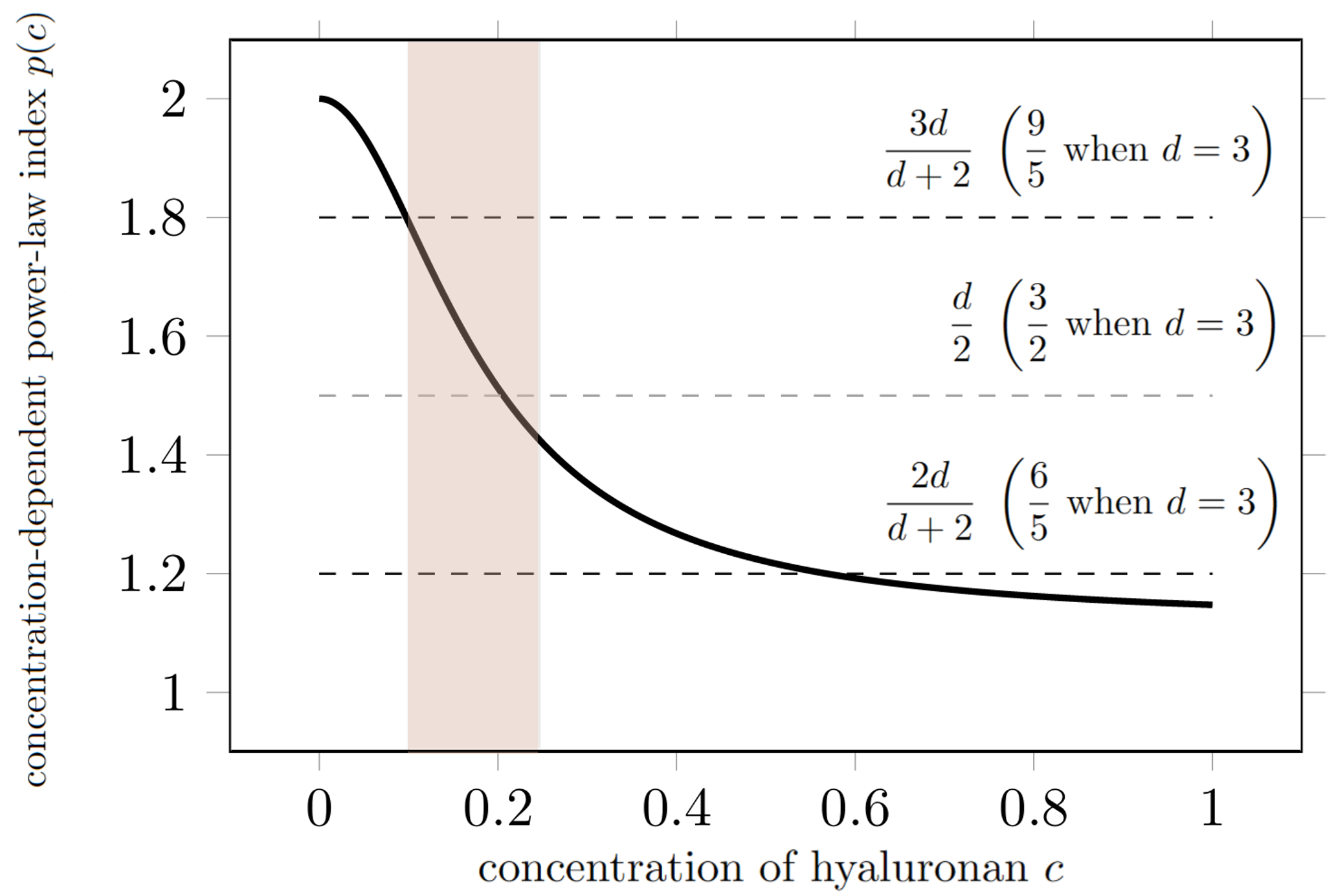}
    \caption{The exponent $p(c)$ is plotted as a function of the concentration of hyaluronan molecules associated with the viscosity \eqref{visc_form} for synovial fluid; see, e.g., \cite{Hron2010} for further details. Physiological values typically observed in non-pathological synovial fluid lie approximately within the range $(0.1, 0.25)$, as indicated by the shaded region in the figure. It is evident that the condition $p^-\geq \frac{d+2}{2}$ ($\frac{5}{2}$ when $d=3$) assumed in \cite{choi2024} is too restrictive, and an existence analysis for smaller values of the power-law index is required to cover the case of higher physiological concentrations.}
    \label{fig:CRF_index}
\end{figure}

The chemically reacting fluid flow model constitutes a generalization of power-law-like non-Newtonian fluid models. The mathematical study of non-Newtonian fluids was initiated in the late 1960s by Ladyzhenskaya and Lions in \cite{Lady1967}, \cite{Lady1969} and \cite{Lion1969}. In these works, the authors performed existence analysis for weak solutions to unsteady problems under the assumption $p \geq \frac{3d+2}{d+2}$ (and $p \geq \frac{3d}{d+2}$ for steady-state problems), utilizing the monotone operator theory. Although the range of admissible exponents was somewhat restrictive for modeling real-world phenomena, these findings nonetheless represented a remarkable advancement in the mathematical theory of incompressible non-Newtonian fluid dynamics. Since then, substantial progress has been made in the theory of non-Newtonian fluids to establish the existence of weak solutions for values of $p$ that are consistent with real-world phenomena. In particular, to control the convective term, effective approximation tools such as $L^\infty$ and Lipschitz truncation techniques were introduced, leading to the establishment of existence results for weak solutions corresponding to a broader class of power-law indices, namely for $p > \frac{2d}{d+2}$; see, for instance, \cite{FMS1997, R1997, FMS2003, DMS2008} for steady-state problems, and \cite{MNJR2001, W2007, Diening2010, BDS2013} for unsteady flows. Further investigation for this model can be found in a number of literature including \cite{Malek2003, Diening2010, kang_2, kang_1}.

Interestingly, experimental studies have shown that in many non-Newtonian fluids, the power-law index $p$ is not constant, but varies with external or internal factors. This model naturally arises in mathematical models of \textit{smart fluids}, such as electro-rheological fluids \cite{rubo}, micro-polar electro-rheological fluids \cite{win-r}, magneto-rheological fluids \cite{bia_2005}, thermo-rheological fluids \cite{AR06}, and chemically reacting fluids \cite{HMPR10}. In each of these models, the variable power-law exponent $p(\cdot)$ depends on certain physical fields or quantities such as electric or magnetic fields, temperature, or the concentration of a specific molecule. Smart fluids of this type have attracted considerable interest due to their wide range of potential applications across diverse fields, including the automotive and heavy machinery sectors, electronics, aerospace engineering, and biomedical technologies (see, e.g., \cite[Chap.\ 6]{smart_fluids} and references therein).

In this paper, among these models, we focus on the chemically reacting fluids where the power-law index depends on the chemical concentration of the fluid. A representative example is synovial fluid, which can be found within joint cavities. In this case, the power-law index depends on the concentration of a specific molecule known as {\textit{hyaluronan} present in the solution. The rheological background of this model can be found in \cite{Hron2010, Pustejovska2012}. For the elliptic problem, the first mathematical results on this model were presented in \cite{Bulicek2013, Bulicek2014}, where the existence of the weak solution for $p^- \ge \frac{d}{2}$ was established. Furthermore, the existence of the classical solutions to the stationary model in two dimensions was shown in \cite{Bulicek2019}. From the viewpoint of computational mathematics, a conforming finite element approximation for the stationary model was constructed, and the convergence analysis was performed in \cite{ko2d, ko3d}. For the corresponding parabolic problem, the existence of weak solutions was investigated in \cite{Ko2022}, where the author established an existence theory under the condition $p^- > \frac{d+2}{2}$ using the monotone operator theory. Moreover, the global-in-time existence of strong solutions in both two and three spatial dimensions, under the assumption $p^- \geq \frac{d+2}{2}$, was shown in \cite{choi2024}.

However, experimental and rheological studies, such as those reported in \cite{Bulicek2009, Hron2010}, indicate that synovial fluid, which is the prototype example of a chemically reacting fluid, exhibits shear-thinning behavior, as can be found in Figure \ref{fig:CRF_index}. Therefore, the model treated in \cite{choi2024} does not fully capture the experimentally observed rheology of the chemically reacting fluids. This motivates us to prove the well-posedness of strong solutions for this model with less restrictive values of $p$. Note that the method employed in \cite{choi2024} does not allow for a further reduction of the lower bound $p^-$ in the global-in-time framework for strong solutions. However, motivated by the approach of Diening et al. \cite{Diening2005}, we identify a new possibility in the local-in-time setting to establish existence results within the shear-thinning regime. This constitutes the primary objective of the present work.

The rest of this paper is organized as follows. In Section \ref{sec:prelim}, we introduce some notation and auxiliary results, and state our main theorems. Section \ref{sec:galerkin} is devoted to the construction of a Galerkin approximation for a regularized system, along with the derivation of uniform estimates for the approximate solutions. In Section \ref{sec:limit}, we establish the local-in-time existence of a strong solution to the chemically reacting fluid flow model by passing to the limit in the approximate system. Section \ref{sec:unique} is concerned with the uniqueness of strong solutions under an additional assumption on $p(\cdot)$. In Section \ref{sec:2D}, 
by applying our proof to the two-dimensional problem, we will establish the same result under less restrictive assumptions in two dimensions. Finally, concluding remarks are presented in Section \ref{sec:conc}.

\section{Preliminaries and main results}\label{sec:prelim}

In this section, we shall introduce our main result along with some preliminaries which will be used throughout the paper. For two vectors $\boldsymbol{a}$, $\boldsymbol{b}$, $\boldsymbol{a}\cdot \boldsymbol{b}$ means the dot product and for two tensors $\mathbb{A}$, $\mathbb{B}$, $\mathbb{A}:\mathbb{B}$ denotes the scalar product. Throughout the paper, the symbol $C$ denotes a generic positive constant, whose value may change from line to line. For $k\in\mathbb{N}\cup\{0\}$ and $1\leq q\leq\infty$, $L^q(\Omega)$ and $W^{k,q}(\Omega)$ present the classical Lebesgue and Sobolev spaces respectively. Moreover, we shall simply write $\|\cdot\|_q=\|\cdot\|_{L^q(\Omega)}$ and $\|\cdot\|_{k,q}=\|\cdot\|_{W^{k,q}(\Omega)}$. For a Banach space $X$, the Bochner space $L^q(I;X)$ is defined by the family of functions with the finite Bochner-type norm
\begin{align*}
    & \|u\|_{L^q(I;X)} := \begin{cases}
        \bigg(\displaystyle \int_I \|u(t)\|_X^q \dt \bigg)^\frac{1}{q} & \text{if } 1 \le q < \infty, \\
        \displaystyle\sup_{t\in I} \|u(t)\|_X & \text{if } q = \infty,
    \end{cases}
\end{align*}
with an abbreviation $L^q(Q_T)=L^q(I;L^q(\Omega))$. Since we will deal with the spatially periodic functions, we shall restrict ourselves to the case of the mean-zero functions over $\Omega$. In this perspective, we will use the following function spaces throughout the paper which are frequently used in the study of incompressible fluid flow problems on a periodic domain: For $k \in \mathbb{N} \cup \{0\}$ and $1 \le q < \infty$,
\begin{align*}
&\mathcal{V}:= \{ \varphi \in C^\infty(\Omega) : \langle \varphi,1\rangle = 0,\,\,{\rm div}\, \varphi =0 \}, \\
&W^{k,q}_{\rm div} (\Omega) := \overline{\mathcal{V}}^{\|\cdot\|_{W^{k,q}}},
\end{align*}
where $\innerproduct{\cdot}{\cdot}$ denotes the $L^2$-inner product over $\Omega$.

Unlike standard power-law-like non-Newtonian fluids, the energy functional $\int_\Omega |\DD\vv|^{p(c)} \dx$ arising in the chemically reacting fluid flow model cannot be adequately described within the framework of classical Lebesgue and Sobolev spaces. Consequently, considerable effort has been devoted to the study of variable-exponent Lebesgue and Sobolev spaces, denoted by $L^{p(\cdot)}$ and $W^{k,p(\cdot)}$, respectively. The structural properties of these spaces play a crucial role in the analysis of the model. For further details, we refer the reader to \cite{diening}.

We also define the notations that will be used frequently throughout the paper. Specifically, we shall introduce the relevant energies for the model defined as:
\begin{equation}\label{special_energy}
\begin{aligned}
\calI_p(c,\vv) := \int_\Omega (\overline{\DD}\vv)^{p(c)-2}|\nabla \DD\vv|^2 \dx, \\
\calJ_p(c,\vv) := \int_\Omega (\overline{\DD}\vv)^{p(c)-2}|\pa_t \DD\vv|^2 \dx,
\end{aligned}
\end{equation}
where
\[
\overline{\DD} \vv := (1+ |\DD\vv|^2)^\frac{1}{2}.
\]
In the above quantities, since $|\nabla^2 \vv| \leq 3|\nabla\DD\vv| \le 6|\nabla^2 \vv|$ in general, the quantity $|\nabla\DD \vv|$ can always be replaced by $|\nabla^2 \vv|$ at the expense of increasing the multiplicative constant. This type of energy functional will frequently appear in our later analysis. In particular, the energy associated with the spatial derivative satisfies the following property, whose proof can be found in \cite{Diening2005}. 
\begin{lemma} \label{DDv,Dv}
     Let $\calI_p(c,\vv)$ be the energy defined in \eqref{special_energy} with variable exponent $p(\cdot)$ satisfying $1<p^-\leq p(\cdot)\le p^+\le2$. Then for sufficiently smooth $\vv$, there holds for a.e. $t\in I$ that
    \begin{equation} 
        \|\nabla^2 \vv(t)\|_{p^-}^{p^-} \le C \calI_p(\cdot,\vv)(t) + C \|\overline{\DD}\vv(t)\|_{p^-}^{p^-},\label{grad^2vp}\\
    \end{equation}
\end{lemma}
Moreover, we will make use of the following inequality in the later uniqueness argument, whose proof can also be found in \cite{Diening2005}.
\begin{lemma} \label{difference_q}
    Let $1 \le \ell \le 2$ and $p(\cdot)$ be a variable exponent satisfying $1<p^-\leq p(\cdot)\le p^+\le2$. Then for sufficiently smooth $\vv_1$ and $\vv_2$, there holds
    \[
    \|\DD(\vv_1 - \vv_2)\|_\ell \le C \left( \int_\Omega ( 1+ |\DD\vv_1|^2 + |\DD\vv_2|^2)^\frac{p(\cdot)-2}{2}|\DD\vv_1 -\DD\vv_2|^2 \dx \right)^\frac{1}{2} \Big\|(\overline{\DD}\vv_1)^\frac{2-p(\cdot)}{2} + (\overline{\DD}\vv_2)^\frac{2-p(\cdot)}{2} \Big\|_{\frac{2\ell}{2-\ell}}.
    \]
\end{lemma}

Finally, let us present a localized version of Gr\"onwall's inequality, which is particularly useful in the analysis of various parabolic problems. A detailed proof can be found in the appendix of \cite{Diening2005}. Although the original statement assumes $\zeta \in C^1([0,T])$, it is noteworthy that it suffices for $\zeta$ to be absolutely continuous.

\begin{lemma}[Local version of Gr\"onwall's inequality] \label{Gronwalllemma}
    For given positive constants $\alpha$ and $c_0$, assume that $\phi \in L^1(0,T)$ with $\phi \ge 0$ a.e. in $[0,T]$ and $\zeta$ be a nonnegative absolutely continuous function satisfying
    \[
    \zeta'(t) \le \phi(t) + c_0\zeta(t)^{1+\alpha}.
    \]
    Then there exists $T'\in [0,T]$ such that for all $t \in [0,T']$, it follows that
    \[
    \zeta(t) \le \Phi(t) + \Phi(t) \left( (1- \alpha c_0 \Phi(t)^\alpha t)^{-\frac{1}{\alpha}} -1 \right),
    \]
    where
    \[
    \Phi(t) := \zeta(0) + \int_0^t \phi(s) \ds.
    \]
\end{lemma}

Now we are ready to state our main theorems concerning the existence and uniqueness of a strong solution of the system under consideration.

\begin{theorem} \label{maintheorem}
Suppose that the variable exponent $p: \R \to \R$ is a Lipschitz continuous function with $\frac{7}{5} < p^- \leq p(\cdot)\leq p^+ \le 2$. Assume further that the initial conditions and the forcing term satisfy the regularity condition
\[
\|\vv_0\|_{2,2} + \|c_0\|_{2,q_0} + \|\fff\|_{L^\infty(I;W^{1,2}(\Omega))} + \|\pa_t \fff\|_{L^2(I;L^2(\Omega))}  \le C
\]
with $q_0\ge4$.
Then, there exists a time $T'\in(0,T)$ such that the system \eqref{main_sys}-\eqref{concentration} admit a strong solution  $(\vv,\pi, c)$ on the interval $I'=[0,T']$ satisfying 
\begin{align*}
    & \vv \in C(I';L^2(\Omega)) \cap L^\infty(I';W_{\rm div}^{1,\frac{12}{5}}(\Omega)) \cap L^{p^-}(I'; W^{2,p^-}(\Omega)),\\
    & c \in C(I';L^2(\Omega)) \cap L^\infty(I';{W^{1,4}(\Omega)}) \cap L^2(I';W^{2,2}(\Omega)),
    \\
    &\partial_t \vv \in L^\infty(I';L^2(\Omega)), \quad \pa_t c \in L^\infty(I';{L^4(\Omega)}), \quad \nabla\pi\in L^\frac{7}{5}(\Omega).
\end{align*}
\end{theorem}

As a second main result, we will also establish the following theorem concerning the uniqueness of the strong solution described in Theorem \ref{maintheorem}.

\begin{theorem} \label{uniqueness theorem}
    Under the assumptions of Theorem \ref{maintheorem}, if the power-law index $p(\cdot)$ further satisfies the condition $p^+<\frac{2p^--2}{2-p^-} p^-$, then the strong solution constructed in Theorem \ref{maintheorem} is unique. 
\end{theorem}

\begin{remark}
    Note that, since we are considering the case $p^->\frac{7}{5}$, the upper bound of $p^+$ in the above theorem satisfies $\frac{2p^--2}{2-p^-}>\frac{28}{15}$. Therefore, the above uniqueness result holds at least for $p^+<\frac{28}{25}$, which corresponds to a physically relevant regime as depicted in Figure \ref{fig:CRF_index}.
\end{remark}

As described above, the primary objective of this paper is the analysis of the local well-posedness of strong solutions in three dimensions. One may wonder about the corresponding result in two dimensions. As will be shown later, our approach can be directly applied to the two-dimensional setting. In this case, the same result can be established under a less restrictive condition on the exponent, specifically within the whole shear-thinning regime $1 < p^- \leq p(\cdot) \leq p^+ \leq 2$. A detailed discussion is provided in Section~\ref{sec:2D}.

\section{Galerkin approximation and uniform estimates}\label{sec:galerkin}

\subsection{Galerkin approximation for the regularized system} \label{best}

In this section, we investigate the Galerkin approximation for the PDE system \eqref{main_sys}-\eqref{concentration}. As will become evident later, our proof requires a certain regularity of the concentration to apply the Gr\"onwall inequality. To ensure this, we introduce a regularized version of the concentration equation, incorporating smoothed initial data using a mollifier. Upon establishing the existence of a solution to the regularized system, the proof is completed by passing to the limit with respect to the regularization parameter. As described above, we consider the system
\begin{equation} \label{regularsystem}
\begin{aligned}
\pa_t\vv_\delta + (\vv_\delta \cdot \nabla) \vv_\delta - {\rm div}\SSS(c_\delta, \DD \vv_\delta)  &= - \nabla \pi_\delta  + \fff && \text{in } Q_T,  \\
{\rm div}\, \vv_\delta &= 0 &&\text{in } Q_T, \\ 
\pa_t c_\delta + {\rm div}(c_\delta \vv_\delta) -\Delta c_\delta  &= 0 &&\text{in } Q_T,
\end{aligned}
\end{equation}
together with the initial conditions $\vv_\delta(\cdot,0) =\vv_0$ and $c_\delta(\cdot,0) =c_{0,\delta}$ in $\Omega$, where $c_{0,\delta} := \eta_\delta * c_0$ and $\eta_\delta$ is standard molifier with $\delta>0$. Let us now consider the Galerkin approximation for the system \eqref{regularsystem}. 
Let $\{\ww_i\}$ be the set of eigenvectors of the Stokes operator and $\{\lambda_i\}$ be the corresponding set of eigenvalues. Since we consider the case of periodic boundary conditions, we can formally write 
\begin{equation} \label{vbasis}
    -\Delta \ww_i = \lambda_i \ww_i \quad \text{in } \Omega.
\end{equation}
Referring to Theorem 4.11 in \cite{malek}, we can choose $\{\ww_i\}_{i=1}^\infty$ which forms a basis of $W^{s,2}_{\rm div}(\Omega)$, where $s>\frac{5}{2}$ so that $W^{s,2}_{\rm div}(\Omega) \hookrightarrow W^{1,\infty}_{\rm div}(\Omega)$. For the concentration equation, we consider the orthonormal basis $\{z_i\}_{i=1}^{\infty}$ of $W^{1,2}(\Omega)$. We aim to construct an approximate solutions $(\vv_\delta^{n,m}, c_\delta^{n,m})$ given by
\begin{equation*}
\vv_\delta^{n,m} = \sum_{i=1}^n \alpha_i^{n,m}(t)\ww_i \quad\text{and}\quad c_\delta^{n,m} = \sum_{j=1}^m \beta_j^{n,m}(t)z_j,
\end{equation*}
which satisfy the following system of ODEs: for all $i=1,\ldots n$ and $j=1,\ldots m$,
\begin{align}
\innerproduct{\pa_t \vv_\delta^{n,m}}{\ww_i} + \innerproduct{
(\vv_\delta^{n,m} \cdot \nabla) \vv_\delta^{n,m}}{\ww_i}
+ \innerproduct {\SSS_\delta^{n,m}}{\DD \ww_i} &= \innerproduct{\fff}{\ww_i}, \label{aux1} \\ 
\innerproduct{\pa_tc_\delta^{n,m}}{z_j} + \innerproduct{\vv_\delta^{n,m} \cdot \nabla c_\delta^{n,m}}{z_j} + \innerproduct{\nabla c_\delta^{n,m}}{\nabla z_j} &=0, \label{aux2}
\end{align}
 where $\SSS_\delta^{n,m} := \SSS(c_\delta^{n,m},\DD\vv_\delta^{n,m})$. Here, the initial conditions are given by
\[
    \vv_\delta^{n,m}(\cdot,0) = P^n\vv_0 \quad\text{and}\quad c_\delta^{n,m}(\cdot,0) = P^mc_{0,\delta},
\]
where $P^n$, $P^m$ represent the orthogonal projection operators onto $A_n:={\rm span}\{\ww_1,\ldots,\ww_n\}$ and $B_m:={\rm span}\{z_1,\ldots,z_m\}$, respectively.
 Thanks to Carath\'eodory's theorem (see, e.g., \cite{Zeidler1990}), it is straightforward to show that the local existence of solutions to the Galerkin approximation \eqref{aux1}-\eqref{aux2}, and the solution can be extended to the whole time interval $(0,T)$ using the typical uniform estimates and iterative argument. 

Next, since our analysis primarily concerns the Navier--Stokes equations and relies on certain properties of the concentration that are valid only at the continuous level, we first take the limit $m\to\infty$ in the convection-diffusion equation for the concentration. To do this, note that the argument used in \cite{Ko2022} allows us to pass to the limit $m \to \infty$, yielding the solutions $\vv_\delta^n=
\sum_{i=1}^n\alpha_i^n(t)\ww_i\in A_n$ and $c_\delta^n\in L^2(I;W^{1,2}(\Omega))$ with $\pa_t c_\delta^n \in L^2(I;W^{-1,2}(\Omega))$ satisfying the following intermediate problem: for all $i=1,\ldots,n$, 
\begin{align}
\innerproduct{\pa_t \vv_\delta^n}{\ww_i} + \innerproduct{
(\vv_\delta^{n} \cdot \nabla) \vv_\delta^{n}}{\ww_i} + \innerproduct {\SSS_\delta^n}{\DD \ww_i} &= \innerproduct{\fff}{\ww_i}, \label{aux11} \\ 
\innerproduct{\pa_tc_\delta^n}{\varphi} + \innerproduct{\vv_\delta^n \cdot \nabla c_\delta^n}{\varphi} + \innerproduct{\nabla c_\delta^n}{\nabla \varphi} &= 0, \label{aux22}
\end{align}
for arbitrary $\varphi \in W^{1,2}(\Omega)$ with $\SSS_\delta^n:= \SSS(c_\delta^n,\DD \vv_\delta^n)$ and the initial conditions
\[
    \vv_\delta^n(\cdot,0) = P^n\vv_0 \quad\text{and}\quad c_\delta^n(\cdot,0) = c_{0,\delta}.
\]
Note further, from the assumptions $\fff \in L^\infty(I;W^{1,2}(\Omega))$ and $\pa_t \fff \in L^2(I;L^2(\Omega))$ in Theorem \ref{maintheorem},  that we can deduce $\alpha_i^n, \pa_t \alpha_i^n \in L^\infty(I)$ and $\pa_t^2 \alpha_i^n \in L^2(I)$, which implies $\vv_\delta^n \in W^{2,2}(I;A_n)$. Thus, we may apply the classical parabolic regularity theory (see, e.g., \cite{evans}) to obtain
\begin{equation} \label{regularity of cn}
c_\delta^n, \pa_t c_\delta^n, \pa_t^2 c^n_\delta \in L^\infty(I;C^\infty(\Omega)). 
\end{equation}
This enables us to write the equation \eqref{aux22} in a pointwise manner as 
\begin{equation} \label{regularequationcn}
    \pa_t c_\delta^n + \vv_\delta^n \cdot \nabla c_\delta^n - \Delta c_\delta^n = 0 \quad \text{a.e. in }Q_T.
\end{equation}

\subsection{A priori estimates} \label{uniform est 3D}

In this section, we shall derive certain {\textit{a priori}} estimates as an intermediate step toward obtaining a uniform estimate via the application of the local Grönwall inequality. For the sake of simplicity, we will drop the indices $n\in\mathbb{N}$ and $\delta>0$. Throughout the arguments in this section, unless otherwise specified, all constants appearing are understood to be independent of $n\in\mathbb{N}$ and $\delta>0$. By multiplying the $i$-th equation of \eqref{aux11} by $\alpha_i^n$ and summing over $i = 1, \ldots, n$, we can derive
\[
\frac{1}{2}\frac{\rm d}{{\rm d}t}\int_{\Omega} |\vv|^2\dx + \int_{\Omega} \SSS(c,\DD\vv):\DD \vv\dx = \int_{\Omega} \fff\cdot \vv\dx,
\]
from which we can obtain the inequality
\begin{equation} \label{best1}
    \sup_{0\le t \le T} \norm{\vv(t)}_2^2 + \int_0^T \int_{\Omega} \left(|\nabla \vv|^{p(c)} + |\SSS(c,\DD\vv)|^{(p(c))'} \right)\dx\dt \le C.
\end{equation}

Now, let us multiply the $i$-th equation in \eqref{aux11} by $\lambda_i \alpha_i^n(t)$ and sum over $i=1,\ldots,n$. Then it follows from \eqref{vbasis} that
\[
\int_{\Omega} \pa_t \vv \cdot (-\Delta \vv) \dx + \int_{\Omega} ((\vv \cdot \nabla) \vv) \cdot (-\Delta \vv) \dx + \int_{\Omega} \SSS(c,\DD \vv):\DD (-\Delta \vv) \dx = \int_{\Omega} \fff \cdot (-\Delta \vv) \dx.
\]
The integration by parts yields
\begin{equation}\label{deltavtest}
    {\rm A_1} + {\rm A_2} \le \int_\Omega |\nabla \fff \cdot \nabla \vv| \dx + \int_\Omega |\nabla \vv|^3 \dx + \int_\Omega |\pa_t \vv \cdot \Delta \vv| \dx,
\end{equation}
where
\begin{equation*} 
{\rm A_1} := \int_{\Omega}  \frac{\pa \SSS(c,\DD)}{\pa \DD}:(\pa_{x_k}\DD\vv\otimes\pa_{x_k}\DD\vv) \dx\quad\text{and}\quad {\rm A_2} :=\int_{\Omega}  \pa_{x_k} c  \left(\frac{\pa\SSS(c,\DD)}{\pa c} : \pa_{x_k}\DD \vv\right)
  \dx.
\end{equation*}
Thanks to \eqref{P1}, we have from \eqref{deltavtest} that
\begin{equation} \label{Ipv}
C\calI_p(c,\vv) \le |{\rm A_2}| + \int_\Omega |\nabla \fff \cdot \nabla \vv| \dx + \int_\Omega |\nabla \vv|^3 \dx + \int_\Omega |\pa_t \vv \cdot \Delta \vv| \dx.
\end{equation}
Next, we shall take the time derivative to \eqref{aux11} and multiply the $i$-th equation by $\frac{\rm {d}}{{\rm{d}}t}\alpha_i^n(t)$. By summing over $i=1,\ldots,n$, we get
\begin{equation} \label{vttest}
\frac{1}{2} \frac{\rm d}{\dt} \|\pa_t \vv\|_2^2 + {\rm B_1}+{\rm B_2} \le \int_\Omega |\pa_t \fff\cdot  \pa_t \vv | \dx + \int_\Omega |\pa_t((\vv \cdot \nabla) \vv) \cdot \pa_t \vv| \dx,
\end{equation}
with
\begin{equation*}
{\rm B_1} := \int_\Omega \frac{\pa \SSS(c,\DD)}{\pa \DD} : (\DD \pa_t \vv \otimes \DD\pa_t \vv) \dx \quad\text{and}\quad{\rm B_2}:= \int_\Omega \pa_t c \left(\frac{\pa \SSS(c,\DD)}{\pa c} : (\DD \pa_t \vv)\right)  \dx.
\end{equation*}
Again, due to \eqref{P1} and \eqref{vttest}, we have
\begin{equation} \label{vt22}
\frac{1}{2} \frac{\rm d}{\dt} \|\pa_t \vv\|_2^2 + C \calJ_p(c,\vv) \le |{\rm B_2}| + \int_\Omega |\pa_t \fff\cdot  \pa_t \vv | \dx + \int_\Omega |\pa_t((\vv \cdot \nabla) \vv) \cdot \pa_t \vv| \dx.
\end{equation}

In the rest of this section, we mainly aim to estimate the terms on the right-hand side of \eqref{Ipv} and \eqref{vt22}. To get a local-in-time solution for $p^+ \le 2$, motivated by the argument used in \cite{Diening2005}, we shall estimate $\frac{\rm d}{\dt}\|\overline{\DD}\vv\|_s^s$ for some $s>1$ in terms of $\epsilon \calJ_p(c,\vv)$. In this perspective, we start with the following lemma, whose proof can be found in \cite{Diening2005}.

\begin{lemma} \label{Dvss}
    Let $1<s<\infty$, then for a.e. $t\in I$ and for sufficiently small $\epsilon>0$, there holds
    \begin{equation} \label{Dvsseqn}
    \frac{\rm d}{\dt}\|\overline{\DD}\vv(t)\|_s^s \le \epsilon \calJ_p(c,\vv)(t) + C \|\overline{\DD}\vv(t)\|_{2s-p^-}^{2s-p^-}.
    \end{equation}
\end{lemma}
For the later use, we shall choose $s=\frac{12}{5}$, which naturally arise from our assumption $p^->\frac{7}{5}$. The above lemma allows us to add $\|\overline{\DD}\vv\|_\frac{12}{5}$ on the right-hand-side when we estimate the quantities in \eqref{Ipv} and \eqref{vt22}.} In this perspective, let us first estimate the terms $\int_\Omega |\nabla \vv|^3 \dx$ and $\int_\Omega \pa_t ((\vv \cdot \nabla ) \vv) \cdot \pa_t v \dx$ in \eqref{Ipv} and \eqref{vt22} respectively in terms of $\|\overline{\DD}\vv\|_\frac{12}{5}$, which are encapsulated in the following lemmas. The proofs for these estimates are presented in \cite{Diening2005}.

\begin{lemma} \label{gradv3}
    There exist constants $R_{1},R_{2}>\frac{12}{5}$, $R_{3}>2$ such that for sufficiently small $\epsilon>0$, there holds
    \begin{align*}
    \|\nabla \vv\|_3^3 &\le C \|\overline{\DD}\vv\|_\frac{12}{5}^{R_{1}} + \epsilon \calI_p(c,\vv) + \epsilon, \\
    \int_\Omega \pa_t((\vv \cdot \nabla) \vv) \cdot \pa_t \vv \dx &\le \epsilon \calJ_p(c,\vv) + C \left(\|\overline{\DD}\vv\|_\frac{12}{5}^{R_{2}}+\|\pa_t \vv\|_2^{R_{3}} +1 \right).
    \end{align*}
\end{lemma}
Next, the estimate for another term $\int_\Omega \pa_t \vv \cdot \Delta \vv \dx$ in \eqref{Ipv} is presented in the following lemma.
\begin{lemma} \label{patvdeltav}
    For $1<p^- \le 2$ and for sufficiently small $\epsilon>0$, there holds
       \[
    \int_\Omega \pa_t \vv \cdot \Delta \vv \dx \le C\|\pa_t \vv\|_2^\frac{8(p^--1)}{5p^--6} \calJ_p(c,\vv)^\frac{2-p^-}{5p^--6} + \epsilon(\calI_p(c,\vv) +1).
    \]
\end{lemma}
Now, it remains to estimate $|{\rm A_2}|$ and $|{\rm B_2}|$ in \eqref{Ipv} and \eqref{vt22} respectively, in terms of $\|\overline{\DD}\vv\|_\frac{12}{5}$. This part constitutes one of the main novel contributions of the present work in comparison with \cite{Diening2005}. The novelty stems from the fact that, in our system, the power-law index $p$ depends on the concentration $c$. As a result, differentiating the Cauchy stress tensor $\SSS$ yields additional terms involving $\nabla c$ and $\pa_tc$, the treatment of which requires further technical considerations. The details are presented in the following lemma.
\begin{lemma} \label{A2B2}
    Let {$q>3$}. Then there exist constants {$R_4,R_5,R_6>q$ and $R_7,R_8 > \frac{12}{5}$} such that for sufficiently small $\epsilon>0$, there hold
    \begin{align}
    &|{\rm{A_2}}| \le C\|\nabla c\|_q^{R_4} + C \|\overline{\DD}\vv\|_\frac{12}{5}^{R_7} + \epsilon \calI_p(c,\vv) + C, \label{A2est}\\
    &{|{\rm{B_2}}| \le C\|\pa_t c\|_q^{R_5} + C\|\nabla c\|_q^{R_6} + C \|\overline{\DD}\vv\|_\frac{12}{5}^{R_8} + \epsilon \calI_p(c,\vv) + \epsilon \calJ_p(c,\vv) + C.} \label{B2est}
    \end{align}
\end{lemma}

\begin{proof}
Let us first estimate $|{\rm A_2}|$. By \eqref{P1} and H\"older's inequality, we have
\begin{equation} \label{sA2}
\begin{split}
    |{\rm A_2}| &\le C \int_\Omega |\nabla c |  |\overline{\DD} \vv|^{p(c)-1} \log (2 + |\DD \vv|) |\nabla \DD \vv| \dx \\
    &= C \int_\Omega |\nabla c| \log(2+|\DD \vv|)  |\overline{\DD }\vv|^\frac{p(c)}{2} \big|\overline{\DD}\vv^\frac{p(c)-2}{2} \nabla \DD \vv  \big| \dx \\
    &\le C \|\nabla c\|_q \| \log(2+ |\DD \vv|) \|_\alpha \| \overline{\DD}\vv^\frac{p(c)}{2} \|_\beta \|  \overline{\DD}\vv^\frac{p(c)-2}{2} \nabla \DD \vv \|_2,
\end{split}    
\end{equation}
where 
\begin{equation} \label{qalphabeta2}
    1 = \frac{1}{q} + \frac{1}{\alpha} + \frac{1}{\beta} + \frac{1}{2}
\end{equation}
with $\alpha>2$ large enough and $\beta>2$ being chosen later. By Young's inequality, we obtain
\begin{equation} \label{A2}
|{\rm{A_2}}| \le C \|\nabla c\|_q^2 \|\log(2+ |\DD\vv|)\|_\alpha^2\|\overline{\DD}\vv^\frac{p(c)}{2}\|_\beta^2 + \epsilon \calI_p(c,\vv).
\end{equation}
Now we shall define $\eta := \overline{\DD}\vv^\frac{p(c)}{2}$, and it follows that
\begin{equation} \label{gradetaL2}
\|\nabla \eta\|_2 \le C \calI_p(c,\vv)^\frac{1}{2} + C\|\nabla c\|_q \|\eta\|_\beta \|\log(2+ |\DD\vv|)\|_\alpha.
\end{equation}
Furthermore, by the Gagliardo--Nirenberg inequality, we have
\begin{equation} \label{eta beta}
\|\eta\|_\beta^2 \le C\|\eta\|_2^{2\theta} \|\nabla \eta\|_2^{2(1-\theta)} + C\|\eta\|_2^2,
\end{equation}
where $\theta = \frac{6-\beta}{2\beta}$ and $\beta \in (2,6)$ is chosen to satisfy \eqref{qalphabeta2}. Inserting \eqref{gradetaL2} to \eqref{eta beta}, we derive
\[
\|\eta\|_\beta^2 \le C \|\eta\|_2^{2\theta} \calI_p(c,\vv)^{1-\theta} + C \|\eta\|_2^{2\theta}\|\nabla c \|_q^{2(1-\theta)}\|\eta\|_\beta^{2(1-\theta)} \|\log(2+|\DD\vv|)\|_\alpha^{2(1-\theta)} + C \|\eta\|_2^2.
\]
Applying Young's inequality yields
\begin{equation} \label{eta beta2}
\|\eta\|_\beta^2 \le C \|\eta\|_2^{2\theta} \calI_p(c,\vv)^{1-\theta} + C \|\eta\|_2^2\|\nabla c \|_q^{\frac{2-2\theta}{\theta}}\|\log(2+|\DD\vv|)\|_\alpha^{\frac{2-2\theta}{\theta}} + C \|\eta\|_2^2.
\end{equation}
Therefore, by inserting \eqref{eta beta2} to \eqref{A2}, we have
\begin{align*}
    |{\rm A_2}| & \le C\|\nabla c\|_q^2\|\log(2+ |\DD\vv|)\|_\alpha^2\|\eta\|_2^{2\theta} \calI_p(c,\vv)^{1-\theta} + C \|\nabla c\|_q^\frac{2}{\theta} \|\log(2+ |\DD\vv|)\|_\alpha^\frac{2}{\theta} \|\eta\|_2^2 \\
    & \hspace{0.3cm} + C \|\nabla c\|_q^2 \|\log(2+ |\DD\vv|)\|_\alpha^2\|\eta\|_2^2 + \epsilon \calI_p(c,\vv).
\end{align*}
Then we apply Young's inequality to get
\[
|{\rm A_2}| \le C(1 + \|\nabla c\|_q^\frac{2}{\theta} \|\log(2+ |\DD\vv|)\|_\alpha^\frac{2}{\theta} ) \|\eta \|_2^2 + \epsilon \calI_p(c,\vv).
\]
Since $p(c)\le p^+ \le 2$, we see that $\|\eta\|_2^2 = \|\overline{\DD}\vv^\frac{p(c)}{2}\|_2^2 \le \|\overline{\DD}\vv\|_2^2 \le C \|\overline{\DD}\vv\|_\frac{12}{5}^2$.
Also, if we use the fact that $\log(2+ |\DD\vv|) \le C (1+ |\DD\vv|)^\frac{12}{5\alpha}$, we obtain 
\[
|{\rm A_2}| \le C \Big(1 + \|\nabla c\|_q^\frac{2}{\theta} \|\overline{\DD}\vv\|_\frac{12}{5}^\frac{24}{5\alpha\theta}\Big)\|\overline{\DD}\vv\|_\frac{12}{5}^2 + \epsilon\calI_p(c,\vv).
\]
Finally, the application of Young's inequality yields the desired result \eqref{A2est}.

Next, we shall estimate the second term $|{\rm B_2}|$. Analogously to \eqref{sA2}-\eqref{A2}, we have
\[
|{\rm B_2}| \le C \|\pa_t c\|_q^2 \|\log(2+ |\DD\vv|)\|_\alpha^2\|\overline{\DD}\vv^\frac{p(c)}{2}\|_\beta^2 + \epsilon \calJ_p(c,\vv),
\]
where $\alpha>2$ large enough and $\beta>2$ chosen as before. By using the same argument as above, we obtain
\begin{align*}
    |{\rm B_2}| & \le C\|\pa_t c\|_q^2\|\log(2+ |\DD\vv|)\|_\alpha^2\|\eta\|_2^{2\theta} \calI_p(c,\vv)^{1-\theta} + C \|\pa_t c\|_q^2  \|\nabla c\|_q^\frac{2-2\theta}{\theta} \|\log(2+ |\DD\vv|)\|_\alpha^\frac{2}{\theta} \|\eta\|_2^2 \\
    & \hspace{0.4cm} + C \|\pa_t c\|_q^2 \|\log(2+ |\DD\vv|)\|_\alpha^2\|\eta\|_2^2 + \epsilon \calJ_p(c,\vv) \\
    & \le C(1+\|\pa_t c\|_q^\frac{2}{\theta} \|\log(2+ |\DD\vv|)\|_\alpha^\frac{2}{\theta}) \|\eta\|_2^2 + C \|\pa_t c\|_q^2 \|\nabla c\|_q^\frac{2-2\theta}{\theta}\|\log(2+ |\DD\vv|)\|_\alpha^\frac{2}{\theta}\|\eta\|_2^2 \\
    & \hspace{0.4cm} + \epsilon \calI_p(c,\vv) + \epsilon\calJ_p(c,\vv).
\end{align*}
As a final step, we use $\|\eta\|_2^2 \le C \|\overline{\DD}\vv\|_\frac{12}{5}^2$ and $\log(2+ |\DD\vv|) \le C (1+ |\DD\vv|)^\frac{12}{5\alpha}$ again, and apply Young's inequality to get the desired result \eqref{B2est}.

\end{proof}

Finally, we can easily estimate the remaining terms regarding the external forcing in \eqref{Ipv} and \eqref{vt22}. Indeed, we shall apply H\"older's inequality, Young's inequality and Korn's inequality (see, e.g., \cite{malek}) to get
\begin{align*}
    \int_\Omega |\nabla\fff \cdot \nabla \vv| \dx &\le C\|\nabla \fff\|_2\|\nabla \vv\|_\frac{12}{5} \le C\|\nabla \fff\|_2^2 +  C\|\nabla \vv\|^2_\frac{12}{5} \le C + C \|\DD\vv\|_\frac{12}{5}^2, \\
    \int_\Omega |\pa_t \fff\cdot  \pa_t \vv | \dx &\le \|\pa_t \fff\|_2\|\pa_t \vv\|_2 \le C \|\pa_t \fff\|_2^2 + C\|\pa_t \vv\|^2_2.
\end{align*}
Thus, we can deduce from \eqref{Ipv} and \eqref{vt22} that
\begin{equation} \label{Ipv2}
 \calI_p(c,\vv) \le C\left(1+ \|\nabla c\|_q^{R_4} + \|\overline{\DD}\vv\|_\frac{12}{5}^{\max\{R_1,R_{7}\}} + \|\pa_t \vv\|_2^\frac{8(p^--1)}{5p^--6} \calJ_p(c,\vv)^\frac{2-p^-}{5p^--6}\right)
\end{equation}
and
\begin{equation} \label{vt22_2}
\frac{\rm d}{dt} \|\pa_t \vv\|_2^2 + C\calJ_p(c,\vv) \le C\left(1+ \|\pa_t \fff\|_2 + \|\pa_t c\|_q^{R_5} + \|\nabla c\|_q^{R_6} + \|\overline{\DD}\vv\|_\frac{12}{5}^{\max\{R_2,R_{8}\}}+\|\pa_t \vv\|_2^{R_{3}}\right) + \epsilon \calI_p(c,\vv).
\end{equation}

Now, based on \eqref{Ipv2} and \eqref{vt22_2}, we make use of Lemma~\ref{Dvss} to derive an auxiliary inequality, which serves as a key step in the application of the local Gr\"onwall inequality. We begin with the following lemma, the proof of which can be found in \cite{Diening2005}.
\begin{lemma} \label{DV24/5}
   There exists a constant $R_{9}>\frac{12}{5}$ such that for sufficiently small $\epsilon>0$,
    \begin{equation} \label{Dv24/5eqn}
    \|\overline{\DD}\vv\|_{\frac{24}{5}-p^-}^{\frac{24}{5}-p^-} \le C\|\overline{\DD}\vv\|_\frac{12}{5}^{R_{9}} + \epsilon(\calI_p(c,\vv)^\frac{5}{3} +1).
    \end{equation}
\end{lemma}
We recall that one of the main advantages of applying Lemma~\ref{Dvss} with \( s = \frac{12}{5} \) is that it allows us to express the subsequent inequalities in terms of \( \|\overline{\DD}\vv\|_{\frac{12}{5}} \). However, this comes at the expense of introducing an unfavorable term, namely \( \|\overline{\DD}\vv\|_{\frac{24}{5} - p^-}^{\frac{24}{5} - p^-} \). To estimate this term, we observe, as shown in Lemma~\ref{DV24/5}, that the right-hand side involves the \( \frac{5}{3} \)-power of the energy functional \( \mathcal{I}_p(c, \vv) \). Accordingly, we raise both sides of \eqref{Ipv2} to the power \( \frac{5}{3} \) and apply Young’s inequality to derive the desired estimate, which is stated in the following lemma.

\begin{lemma}
 There exist constants $R_{10}>q$, $R_{11}>\frac{12}{5}$ and $R_{12}>2$ such that for sufficiently small $\epsilon>0$,
\begin{equation} \label{Ip4/3}
\calI_p(c,\vv)^\frac{5}{3} \le C(1+ \|\nabla c\|_q^{R_{10}} + \|\overline{\DD}\vv\|_\frac{12}{5}^{R_{11}} + \|\pa_t \vv\|_2^{R_{12}}) + \epsilon \calJ_p(c,\vv).
\end{equation}
\end{lemma} 
As we can see in Lemma \ref{A2B2}, the main difference compared to the equations considered in \cite{Diening2005} is the appearance of $\nabla c$ and $\pa_t c$ terms in a series of estimates. To handle these terms, we will estimate $\|\nabla c\|_q$ and $\|\pa_t c\|_q$ by testing some suitable quantities to the equation \eqref{regularequationcn}. These are encapsulated in the following two consecutive lemmas.
\begin{lemma} \label{gradcqlemma}
    Let $q \ge 2$. Then there exist constants $R_{13}>\frac{12}{5}$, $R_{14}>q$ such that
    \begin{equation*} 
        \frac{\rm d}{\dt} \|\nabla c \|_q^q + C\int_\Omega \left|\nabla (|\nabla c |^\frac{q}{2}) \right|^2 \dx \le C \|\overline{\DD}\vv\|_\frac{12}{5}^{R_{13}} + C \|\nabla c \|_q^{R_{14}} + C.
    \end{equation*}
\end{lemma}

\begin{proof}

    If we multiply $- \nabla \cdot (|\nabla c|^{q-2} \nabla c)$ to \eqref{regularequationcn} and integrate over $\Omega$, then the integration by parts yields 
    \[
    \frac{\rm d}{{\rm d}t} \int_\Omega |\nabla c |^q \dx + C\int_\Omega \left|\nabla (|\nabla c |^\frac{q}{2}) \right|^2 \dx \le C \int_\Omega |\nabla \vv| |\nabla c|^q \dx.
    \]
    Moreover, by H\"older's inequality, we note that
    \[
    \int_\Omega |\nabla \vv| |\nabla c |^q \dx \le \|\nabla \vv\|_\frac{12}{5} \| |\nabla c |^\frac{q}{2} \|_{\frac{24}{7}}^2.
    \]
    For the last term, we use the Gagliardo--Nirenberg inequality (see, e.g., \cite{niren}) to get
    \[
    \||\nabla c |^\frac{q}{2}\|_{\frac{24}{7}}^2 \le C\| |\nabla c|^\frac{q}{2}\|_2^{\frac{3}{4}} \|\nabla (|\nabla c|^\frac{q}{2})\|_2^{\frac{5}{4}} + C\||\nabla c|^\frac{q}{2}\|_2^2.
    \]
    Then the desired estimate follows by applying Young's inequality and Korn's inequality.
\end{proof}

\begin{lemma} \label{patclemma}
    Let $q\ge4$. Then there exist constants $R_{15}>2$, $R_{16},R_{17}>q$ such that
    \begin{equation*} 
    \frac{\rm d}{\dt} \|\pa_t c\|_q^q + C\int_\Omega \left| \nabla (|\pa_t c|^\frac{q}{2}) \right|^2 \dx \le C \|\pa_t \vv\|_2^{R_{15}} + C\|\nabla c\|_q^{R_{16}} + C\|\pa_t c\|_q^{R_{17}} +C.
    \end{equation*}
\end{lemma}

\begin{proof}
    We shall take the time derivative to \eqref{regularequationcn} and test $|\pa_t c|^{q-2} \pa_t c$ to get
    \[
    \frac{1}{q} \frac{\rm d}{\dt} \|\pa_t c\|_q^q + C \int_\Omega \left| \nabla(|\pa_t c |^\frac{q}{2}) \right|^2 \dx = - \int_\Omega (\pa_t \vv \cdot \nabla c) |\pa_t c|^{q-2} \pa_t c \dx.
    \]
    Since $\frac{1}{2} + \frac{1}{q} + \frac{q-1}{3q} \le 1$ when $q \ge 4$, by H\"older's inequality and the Sobolev embedding, we have
    \begin{align*}
        \bigg|\int_\Omega (\pa_t \vv \cdot \nabla c) |\pa_t c|^{q-2} \pa_t c \dx\bigg|
        & \le \int_\Omega |\pa_t \vv| |\nabla c| |\pa_t c|^{q-1} \dx \le C \|\pa_t \vv\|_2\|\nabla c\|_q\|\pa_t c\|_{3q}^{q-1} +C\\
        & = C \|\pa_t \vv\|_2 \|\nabla c \|_q \||\pa_t c|^\frac{q}{2} \|_6^\frac{2(q-1)}{q} +C\\
        &\le C \|\pa_t \vv\|_2 \|\nabla c \|_q \bigg(\|\nabla (|\pa_t c|^\frac{q}{2})\|_2^\frac{2(q-1)}{q} +  \||\pa_t c|^\frac{q}{2}\|_2^{\frac{2(q-1)}{q}} \bigg) + C.
    \end{align*}
    Since $\frac{2(q-1)}{q} <2$, we can apply Young's inequality to get the desired result.
\end{proof}

Now we choose $q = q_0$, where $q_0$ is given in Theorem~\ref{maintheorem}. By reinstating the index $n \in \mathbb{N}$ and $\delta > 0$, and combining the estimates \eqref{Dvsseqn} with $s = \frac{12}{5}$, \eqref{Ipv2}, \eqref{vt22_2}, \eqref{Dv24/5eqn}, and \eqref{Ip4/3}, we obtain the following {\textit{a priori}} estimate, which will be used in the application of the local Gr\"onwall inequality: for some constants $S_1>\frac{12}{5}$, $S_2>2$ and $S_3,S_4>q$, there holds
\begin{equation} \label{Gronwall}
\begin{split}
    \frac{\rm d}{\dt} \|&\overline{\DD}\vv_\delta^n\|_\frac{12}{5}^\frac{12}{5} + \frac{\rm d }{\dt}  \|\pa_t \vv_\delta^n\|_2^2 + \frac{\rm d}{\dt}\|\nabla c_\delta^n\|_q^q + \frac{\rm d}{\dt}\|\pa_t c_\delta^n\|_q^q  + C\calI_p(c_\delta^n,\vv_\delta^n) + C \calI_p(c_\delta^n,\vv_\delta^n)^\frac{5}{3} + C \calJ_p(c_\delta^n,\vv_\delta^n) \\
    & \le C\Big( 1 + \|\pa_t \fff\|_2^2+ \|\overline{\DD}\vv_\delta^n\|_\frac{12}{5}^{S_1} + \|\pa_t \vv_\delta^n\|_2^{S_2}+ \|\nabla c_\delta^n\|_q^{S_3} + \|\pa_t c_\delta^n\|_q^{S_4} \Big).
\end{split}
\end{equation}

\subsection{Application of local Gr\"onwall inequality}
The objective of this section is to obtain the uniform boundedness for $\vv^n_\delta$ and $c_\delta^n$ by applying the localized version of Gr\"onwall's lemma (Lemma \ref{Gronwalllemma}). We present this part as a separate section, since our setting involves some additional technical considerations, compared to other related works, e.g., \cite{Diening2005}. To be more specific, to apply the local Gr\"onwall inequality, we firstly need to check that $\|\nabla c_\delta^n(t)\|_q$ and $\|\pa_t c_\delta^n(t)\|_q$ are absolutely continuous with respect to the temporal variable $t\in I$. As in \cite{Diening2005}, the velocity field lies at the Galerkin level and the spatial norm is sufficiently smooth in time, allowing the direct application of the Grönwall inequality. For the case of the concentration, however, we employed a two-level Galerkin approximation, in which we first let $m \to \infty$ to pass the convection-diffusion equation to the continuous level. Consequently, the required regularity for $ \|\nabla c_\delta^n(t)\|_q $ and $\|\partial_t c_\delta^n(t)\|_q$ must be verified explicitly. This can be deduced from the preceding discussion \eqref{regularity of cn}, where we exploited the parabolic regularity theory. This is precisely the reason we regularized the initial condition for the concentration. Without such regularization, one would have to impose higher differentiability assumptions on the initial data. Although this regularization introduces an additional limiting procedure and complicates the analysis, it enables us to derive the necessary properties under minimal assumptions on $c_0$.

Secondly, we also need to show that the initial data $\vv_\delta^n$ and $c_\delta^n$ are independent of $n\in\mathbb{N}$ and $\delta>0$.
Indeed, note here that
\[
\|(\nabla \vv_\delta^n)(0)\|_\frac{12}{5} = \|\nabla P^n \vv_0 \|_\frac{12}{5} \le C \|P^n\vv_0\|_{2,2} \le C\|\vv_0\|_{2,2} \le C,
\]
and from the continuity of $\nabla c_\delta^n(t)$ on $I$, we see that
\[
\|(\nabla c_\delta^n)(0)\|_q = \| c_{0,\delta}\|_{1,q} \le C.
\]
Next, for the velocity term $\pa_t \vv_\delta^n$, with $\varphi \in L^2(\Omega)$ satisfying $\|\varphi\|_2 \le 1$, we have
\begin{align}
    |\innerproduct{\pa_t \vv_\delta^{n}(0)}{\varphi}| &= \lim_{t \to 0+}|\innerproduct{\pa_t \vv_\delta^{n}(t)}{\varphi}| \nonumber= \lim_{t \to 0+} |\innerproduct{\pa_t\vv_\delta^n(t)}{P^n\varphi}| \nonumber \\
    &= \lim_{t \to 0+} |\innerproduct{{\rm div}\, \SSS(c^{n}_\delta(t),\DD\vv^{n}_\delta(t)) + (\vv^{n}_\delta(t)\cdot\nabla)\vv^{n}_\delta(t) - \fff(t)}{P^n\varphi}| \nonumber\\
    &\le \lim_{t \to 0+} \|\nabla \SSS(c^{n}_\delta(t),\DD\vv^{n}_\delta(t))\|_2 +  \lim_{t \to 0+}\| \vv^{n}_\delta(t)\|_{2,2}^2 + \lim_{t \to 0+} \|\fff(t)\|_2. \label{gradS}
\end{align}
To estimate the first term of \eqref{gradS}, from \eqref{P1}, we see that the Cauchy stress tensor $\SSS(\cdot,\cdot)$ satisfies the estimate
\[
|\nabla \SSS(c,\DD\vv)| \le C (1+ |\DD\vv|^2)^\frac{p(c)-2}{2} |\nabla \DD\vv| + C (1+ |\DD\vv|^2)^\frac{p(c)-1}{2} \log(2+|\DD\vv|)|\nabla c|=:{\rm E_1} + {\rm E_2}.
\]
Since $p^+ \le 2$, it follows that
\[
{\rm E_1} \le C |\nabla \DD\vv|.
\]
For the second term ${\rm E_2}$, we have that
\begin{align*}
{\rm E_2} &= C (1+ |\DD\vv|^2)^\frac{p(c)-2}{2} |\overline{\DD}\vv| \log(2+ |\DD\vv|)|\nabla c|\\
& \le C|\overline{\DD}\vv|^{3}+ C|\nabla c|^2 \le C + C|\DD\vv|^{3} + C|\nabla c|^2,
\end{align*}
where we have used the condition $p^+\le2$ once more and the fact $\log(2+|\DD\vv|) \le C(1+|\DD\vv|)^{\frac{1}{2}}$, together with Young's inequality. Therefore, it follows from the Sobolev embedding that
\[
\begin{aligned}
\lim_{t \to 0+}\|\nabla \SSS(c^{n}_\delta(t),\DD\vv^{n}_\delta(t))\|_2 &\le  C+ C \lim_{t \to 0+}( \|\vv^{n}_\delta(t)\|^3_{2,2} + \|\nabla c^{n}_\delta(t)\|^2_4) \\
&=C + C( \|\vv^{n}_\delta(0)\|^3_{2,2} + \|\nabla c^{n}_\delta(0)\|^2_4) \\
&=C+ C(\|P^n\vv_0\|^3_{2,2} + \| c_{0,\delta}\|^2_{1,4}) \le C.
\end{aligned}
\]
For the second term in \eqref{gradS}, as above, it is straightforward to see that $\lim_{t \to 0+} \|\vv^n_\delta(t)\|_{2,2}^2 = \|\vv^n_\delta(0)\|_{2,2}^2 \le C$. Moreover, since 
$\fff \in L^\infty(I;W^{1,2}(\Omega)$ and $\pa_t \fff \in L^2(I;L^2(\Omega))$ implies $\fff \in C(I;L^2(\Omega))$, we see that $\lim_{t \to 0+}\|\fff(t)\|_2 =\|\fff(0)\|_2 \le C$. Therefore, we have $|\innerproduct{\pa_t \vv_\delta^n(0)}{\varphi}| \le C$, which imiplies $\|\pa_t\vv_\delta^n(0)\|_2 \le C$. Finally, for the terms concerning the concentrations, since $\|\nabla^2 c^n_\delta\|_q$ is continuous on $I$, thanks to the regularity \eqref{regularity of cn}, there holds that
\begin{align*}
    \|\pa_t c^n_\delta(0)\|_q &= \lim_{t\to 0+} \|\pa_t c^n_\delta(t)\|_q= \lim_{t \to 0+} \|-\vv^n_\delta(t) \cdot \nabla c^n_\delta(t) + \Delta c^n_\delta(t)\|_q \\
    &\le \lim_{t \to 0+} (\|\vv^n_\delta(t)\|_\infty \|\nabla c^n_\delta(t)\|_q + \|\nabla^2 c^n_\delta(t)\|_q ) \\
    &= \|\vv^n_\delta(0)\|_\infty \|\nabla c^n_\delta(0)\|_q + \|\nabla^2 c^n_\delta(0)\|_q \\
    &\le \|P^n\vv_0\|_{2,2} \|c_{0,\delta}\|_{1,q} + \|c_{0,\delta}\|_{2,q} \le C.
\end{align*}

Now we are ready to apply the local Gr\"onwall inequality. From the assumption of the main theorem, we know that $\|\pa_t\fff\|_2^2 \in L^1(I)$. This, together with the above discussions, enables us to apply Lemma \ref{Gronwalllemma} to \eqref{Gronwall}. Then, by Korn's inequality, we can finally obtain the following uniform estimates:
\begin{equation}\label{uniformest1}
\|\nabla\vv^n_\delta\|_{L^\infty(I';L^\frac{12}{5}(\Omega))} + \|\pa_t \vv^n_\delta\|_{L^\infty(I';L^2(\Omega))} + \|\nabla c^n_\delta\|_{L^\infty(I';L^q(\Omega))} + \|\pa_t c^n_\delta\|_{L^\infty(I';L^q(\Omega))} \le C,
\end{equation}
where $I'=[0,T']$ for some $T' \in (0,T)$ and the constant $C$ on the right-hand side of \eqref{uniformest1} is independent of $n\in\mathbb{N}$ and $\delta>0$. Furthermore, we can also have a uniform estimate for the energies: for a constant independent of $n\in\mathbb{N}$ and $\delta>0$, there holds
\begin{equation} \label{uniformest2}
\|\calI_p(c^n_\delta,\vv^n_\delta)\|_{L^\frac{5}{3}(I')} + \|\calJ_p(c^n_\delta,\vv^n_\delta)\|_{L^1(I')} \le C.
\end{equation}

\section{Proof of Theorem \ref{maintheorem}}\label{sec:limit}
In this section, we establish the main existence theorem for strong solutions by performing limiting processes based on the uniform estimates derived in the previous section. To begin with, by combining \eqref{grad^2vp}, \eqref{uniformest1}, and \eqref{uniformest2}, we observe that
\begin{equation}\label{v2,pest}
\| \vv^n_\delta\|_{L^{p^-}(I';W^{2,p^-}(\Omega))} \le C,
\end{equation}
where the constant $C$ is independent of $n\in\mathbb{N}$ and $\delta>0$. Therefore, together with the uniform estimate \eqref{uniformest1}, we can extract a (not relabeled) subsequence with respect to $n\in\mathbb{N}$ such that
\begin{equation} \label{convergences}
    \begin{aligned}
        \vv_\delta^n &\rightharpoonup \vv_\delta &&\mbox{weakly in }L^{p^-}(I';W^{2,p^-}(\Omega)),  \\
     \vv_\delta^n & \overset{\ast} {\rightharpoonup} \vv_\delta && \text{weakly-$^*$ in } L^\infty(I';W^{1,\frac{12}{5}}(\Omega)), \\
     \partial_t \vv_\delta^n & \overset{\ast} {\rightharpoonup} \partial_t \vv_\delta && \text{weakly-$^*$ in } L^\infty(I';L^2(\Omega)), \\
     c_\delta^n & \overset{\ast}  {\rightharpoonup} c_\delta && \text{weakly-$^*$ in } L^\infty(I';W^{1,q}(\Omega)), \\
     \pa_t c_\delta^n & \overset{\ast} {\rightharpoonup} \pa_t c_\delta && \text{weakly-$^*$ in } L^\infty(I';L^q(\Omega)). 
    \end{aligned}
\end{equation}
Since $W^{2,p^-}(\Omega)$ is compactly embedded into $W^{1,2}(\Omega)$ for $p^->\frac{6}{5}$, from the Aubin--Lions lemma (see, e.g., \cite{Lion1969}), we obtain that 
\begin{equation} \label{conv_v}
    \vv_\delta^n \to \vv_\delta \quad \mbox{strongly in }L^2(I';W^{1,2}(\Omega)),
\end{equation}
from which we can deduce
\begin{equation*}
    \DD\vv_\delta^n \to \DD\vv_\delta \quad \text{a.e. in }Q_{T'}.
\end{equation*}
For the approximate concentration, applying the Aubin--Lions lemma once more, together with the compact embedding $W^{1,q}(\Omega) \hookrightarrow\hookrightarrow L^q(\Omega)$ yields 
\begin{equation*}
    c_\delta^n \to c_\delta \quad \text{strongly in } L^q(Q_{T'}),
\end{equation*}
which implies that
\begin{equation*}
    c_\delta^n \to c_\delta \quad \text{a.e. in }Q_{T'}.
\end{equation*}
Therefore, the continuity of $\SSS(\cdot,\cdot)$ yields that
\begin{equation*}
    \SSS(c_\delta^n,\DD\vv_\delta^n) \to \SSS(c_\delta,\DD\vv_\delta) \quad \text{a.e. in }Q_{T'}.
\end{equation*}
In addition, it follows from \eqref{best1} that
\begin{equation*}
    \|\SSS(c_\delta^n,\DD\vv_\delta^n)\|_{L^1(Q_{T'})} \le C,
\end{equation*}
where the constant $C$ is independent of $n\in\mathbb{N}$ and $\delta>0$. Thus, by applying Vitali's convergence theorem (see, e.g., \cite{malek}), we get
\begin{equation} \label{convergence of S}
    \SSS(c_\delta^n,\DD\vv_\delta^n) \to \SSS(c_\delta,\DD\vv_\delta) \quad \text{in }L^1(Q_{T'}),
\end{equation}
and hence, we have identified the limit of the Cauchy stress tensor. For the approximate convective term, from the convergence \eqref{conv_v}, we have
\begin{equation} \label{convergence of convective}
    (\vv_\delta^n \cdot \nabla) \vv_\delta^n \to (\vv_\delta \cdot \nabla) \vv_\delta \quad \mbox{in } L^\frac{7}{5}(Q_{T'}).
\end{equation}
Now we multiply $\xi \in C_0^\infty(I')$ to \eqref{aux11} and integrate over $I'$ to get
\begin{equation}\label{disc_inter_eq}
    \int_{I'} \big(\innerproduct{\pa_t \vv^n_\delta + (\vv^n_\delta \cdot \nabla) \vv^n_\delta}{\ww_i} + \innerproduct{ \SSS(c^n_\delta,\DD\vv^n_\delta)}{\DD\ww_i} \big) \xi \dt
    = \int_{I'} \innerproduct{\fff}{\ww_i}\, \xi  \dt.
\end{equation}
Then from \eqref{convergences}, \eqref{convergence of S} and \eqref{convergence of convective}, taking the limit $n \to \infty$ to the discrete equations \eqref{disc_inter_eq} yields
\begin{equation} \label{v equation1}
    \int_{I'} \big(\innerproduct{\pa_t \vv_\delta + (\vv_\delta \cdot \nabla) \vv_\delta}{\ww_i} + \innerproduct{ \SSS(c_\delta,\DD\vv_\delta)}{\DD\ww_i} \big) \xi \dt
    = \int_{I'} \innerproduct{\fff}{\ww_i} \,\xi  \dt.
\end{equation}
Next, in order to write the equations involved in \eqref{v equation1} in a strong form, we need an additional regularity of the Cauchy stress tensor $\SSS(\cdot,\cdot)$. To do this, note from \eqref{P1} that there holds
\begin{align*}
|\nabla \SSS(c_\delta,\DD\vv_\delta)|^\frac{7}{5} &\le C (1+|\DD\vv_\delta|^2)^{\frac{7}{5}\frac{p(c_\delta)-2}{2}} |\nabla \DD\vv_\delta|^\frac{7}{5} + C(1+ |\DD\vv_\delta|^2)^{\frac{7}{5}\frac{p(c_\delta)-1}{2}} \log^\frac{7}{5}(2+ |\DD\vv_\delta|)|\nabla c_\delta|^\frac{7}{5} \\
&\le C|\nabla \DD\vv_\delta|^\frac{7}{5} + C(1+|\DD\vv_\delta|^2)^\frac{7}{10} \log^\frac{7}{5}(2+ |\DD\vv_\delta|)|\nabla c_\delta|^\frac{7}{5}.
\end{align*}
Since $q \ge 4$, we apply Young's inequality to get
\[
|\nabla \SSS(c_\delta,\DD\vv_\delta)|^\frac{7}{5} \le C |\nabla \DD\vv_\delta|^\frac{7}{5} + C|\DD\vv_\delta|^\frac{12}{5} + C\log^\alpha(2+ |\DD\vv_\delta|) + C|\nabla c_\delta |^q + C,
\]
where $\frac{7}{12} + \frac{7}{5\alpha} + \frac{7}{5q} =1$ with $\alpha>1$ being large enough. Then together with $\log(2+ |\DD\vv|) \le C(1+ |\DD\vv|)^\frac{12}{5\alpha}$, we write
\[
|\nabla \SSS(c_\delta,\DD\vv_\delta)|^\frac{7}{5} \le C |\nabla \DD\vv_\delta|^\frac{7}{5} + C|\DD\vv_\delta|^\frac{12}{5} + C|\nabla c_\delta |^q + C.
\]
Consequently, using \eqref{convergences}, we deduce that
\begin{equation}\label{S_inter_reg}
\|\nabla \SSS(c_\delta,\DD\vv_\delta)\|_{L^\frac{7}{5}(Q_{T'})}\le C.
\end{equation}
Summarizing the above discussions, we obtain the following uniform estimate: for some constant $C$ independent of $\delta>0$, there holds
\begin{equation}\label{all_inter_reg}
    \|\pa_t\vv_\delta\|_{L^\infty(I';L^2(\Omega))} + \|\nabla \SSS(c_\delta,\DD\vv_\delta)\|_{L^\frac{7}{5}(Q_{T'})} + \|(\vv_\delta \cdot \nabla) \vv_\delta\|_{L^\frac{7}{5}(Q_{T'})} \le C.
\end{equation}
Now, thanks to \eqref{all_inter_reg}, we can rewrite the formulation \eqref{v equation1} as follows: for $\ww \in \{\ww_i\}_{i=1}^\infty$,  
\begin{equation*}
    \int_{I'} \big(\innerproduct{\pa_t \vv_\delta + (\vv_\delta \cdot \nabla) \vv_\delta - {\rm div}\, \SSS(c_\delta,\DD\vv_\delta)}{\ww} \big) \xi \dt
    = \int_{I'} \innerproduct{\fff}{\ww} \,\xi \dx \dt.
\end{equation*}
Therefore, by the density of smooth functions, we may write the equation as
\begin{equation} \label{vestrong}
    \int_{Q_{T'}} \big(\pa_t \vv_\delta+ (\vv_\delta \cdot \nabla) \vv_\delta   - {\rm div} \, \SSS(c_\delta,\DD\vv_\delta) \big) \cdot \ffpsi \dx\dt = \int_{Q_{T'}} \fff \cdot \ffpsi \dx \dt \quad \text{for all }\ffpsi \in C_0^\infty(I';\mathcal{V}).
\end{equation}

Next, let us move to the equation for the concentration. From Lemma \ref{gradcqlemma}, we have $\|c^n_\delta\|_{L^2(I';W^{2,2}(\Omega))} \le C$, which implies
\begin{equation}\label{c W22 conv}
    c^n_\delta \rightharpoonup c_\delta \quad \text{weakly in } L^2(I';W^{2,2}(\Omega)).
\end{equation}
Therefore, combining \eqref{convergences} and \eqref{c W22 conv}, the Aubin--Lions lemma with $W^{2,2}(\Omega) \hookrightarrow\hookrightarrow W^{1,4}(\Omega)$ implies
\begin{equation} \label{conv_c_2}
c^n_\delta \to c_\delta \quad \text{strongly in } L^2(I';W^{1,4}(\Omega)).
\end{equation}
Thus, from \eqref{convergences}, \eqref{conv_v} and \eqref{conv_c_2}, it is straightforward to see that
\begin{equation} \label{v grad c conv}
\vv^n_\delta \cdot \nabla c^n_\delta \to \vv_\delta \cdot \nabla c_\delta \quad \text{in }L^2(Q_{T'}).
\end{equation}
Then, as before, we multiply $\zeta \in C^\infty_0(I')$ to \eqref{aux22} and integrate over $I'$ to obtain
\[
\int_{I'} (\innerproduct{\pa_t c^n_\delta + \vv^n_\delta \cdot \nabla c^n_\delta}{\varphi} + \innerproduct{\nabla c^n_\delta}{\nabla \varphi} ) \zeta \dt =0.
\]
Thanks to \eqref{convergences}, \eqref{c W22 conv} and \eqref{v grad c conv}, and by the density of smooth functions, we may write
\begin{equation} \label{cestrong}
    \int_{Q_{T'}} \big(\pa_t c_\delta  + {\rm div}\,  (c_\delta\vv_\delta)  - \Delta c_\delta  \big)\, \phi\dx\dt = 0 \quad \text{for all }\phi \in C_0^\infty(I';C^\infty(\Omega)),
\end{equation}
with the uniform estimate
\begin{equation*}
    \|\pa_t c_\delta\|_{L^2(Q_{T'})} + \|\, {\rm div} \, (c_\delta \vv_\delta)\|_{L^2(Q_{T'})} + \|\Delta c_\delta\|_{L^2(Q_{T'})} \le C.
\end{equation*}

Next, we need to verify $\vv_\delta(0)= \vv_0$ and $c_\delta(0) = c_{0,\delta}$. We first observe, from $\vv_\delta, c_\delta \in L^2(I';W^{1,2}(\Omega))$ and $\pa_t \vv_\delta, \pa_t c_\delta \in L^2(I';L^2(\Omega))$, that $\vv_\delta, c_\delta \in C(I';L^2(\Omega))$, so that the evaluation at $t=0$ makes sense. To see $\vv_\delta(0)= \vv_0$, we integrate \eqref{aux11} over $t \in (0,\tau)$ and pass $n \to \infty$ to get
\begin{equation} \label{vdelta0}
\int_\Omega \vv_\delta(\tau) \cdot \ww_i \dx + \int_0^\tau \int_\Omega (( \vv_\delta \cdot \nabla) \vv_\delta  - {\rm div}\, \SSS(c_\delta, \DD\vv_\delta) -\fff) \cdot \ww_i \dx \dt = \int_\Omega \vv_0 \cdot \ww_i \dx,
\end{equation}
where we have used \eqref{convergences}, \eqref{convergence of S}, \eqref{convergence of convective} and \eqref{S_inter_reg}, together with the fact $\|P^n \vv_0 - \vv_0\|_2 \to 0$. Since $(( \vv_\delta \cdot \nabla) \vv_\delta  - {\rm div}\, \SSS(c_\delta, \DD\vv_\delta) -\fff) \cdot \ww_i\in L^1(I')$, we get
\[
\int_\Omega \vv_\delta(\tau) \cdot \ww_i \dx \to \int_\Omega \vv_0 \cdot \ww_i \dx \quad \text{as }\tau \to 0+,
\]
which, together with $\vv_\delta \in C(I';L^2(\Omega))$, implies $\vv_\delta(\cdot, 0) = \vv_0$. Similarly, to identify $c_\delta(0) = c_{0,\delta}$, we integrate \eqref{aux22} over $t \in (0,\tau)$ and let $n \to \infty$ to obtain
\begin{equation} \label{cdelta0}
\int_\Omega c_\delta(\tau) \varphi \dx + \int_0^\tau \int_\Omega (\vv_\delta \cdot \nabla c_\delta - \Delta c_\delta) \varphi \dx\dt = \int_\Omega c_{0,\delta} \varphi \dx,
\end{equation}
where we have used \eqref{convergences}, \eqref{c W22 conv} and \eqref{v grad c conv}. Since $(\vv_\delta \cdot \nabla c_\delta - \Delta c_\delta) \varphi \in L^1(I')$, we see that
\[
\int_\Omega c_\delta(\tau) \varphi \dx \to \int_\Omega c_{0,\delta} \varphi \dx  \quad \text{as }\tau \to 0+.
\]
Again, by the fact $c_\delta \in C(I';L^2(\Omega))$, we can conclude that $c_\delta(\cdot, 0) = c_{0,\delta}$.

Now, by weak and weak-* lower semicontinuity of norms, we have from \eqref{v2,pest} and \eqref{uniformest1} that
\[
\|\vv_\delta\|_{L^{p^-}(I';W^{2,p^-}(\Omega))} + \|\nabla \vv_\delta\|_{L^\infty(I';L^\frac{12}{5}(\Omega))} + \|\pa_t \vv_\delta\|_{L^\infty(I';L^2(\Omega))} + \|\nabla c_\delta\|_{L^\infty(I';L^q(\Omega))} + \|\pa_t c_\delta\|_{L^\infty(I';L^q(\Omega))} \le C.
\]
Note here that the constant on the right-hand side of \eqref{uniformest1} is independent of $\delta>0$, and thus, the constant in the above inequality is also independent of $\delta>0$. Therefore, with the similar argument used for the limit $n \to \infty$, we can extract the subsequence (not relabeled) with respect to $\delta>0$ such that
\[
    \begin{aligned}  
        \vv_\delta &\rightharpoonup \vv &&\mbox{weakly in }L^{p^-}(I';W^{2,p^-}(\Omega)), \\
     \vv_\delta & \overset{\ast} {\rightharpoonup} \vv && \text{weakly-$^*$ in } L^\infty(I';W^{1,\frac{12}{5}}(\Omega)), \\
     \partial_t \vv_\delta & \overset{\ast} {\rightharpoonup} \partial_t \vv && \text{weakly-$^*$ in } L^\infty(I';L^2(\Omega)), \\
     c_\delta & \overset{\ast}  {\rightharpoonup} c && \text{weakly-$^*$ in }L^\infty(I';W^{1,q}(\Omega)), \\
     \pa_t c_\delta & \overset{\ast} {\rightharpoonup} \pa_t c && \text{weakly-$^*$ in } L^\infty(I';L^q(\Omega)).
    \end{aligned}
\]
Moreover, we can obtain the estimate
\[
    \|\pa_t\vv\|_{L^\infty(I';L^2(\Omega))} + \|\nabla \SSS(c,\DD\vv)\|_{L^\frac{7}{5}(Q_{T'})} + \|(\vv \cdot \nabla) \vv\|_{L^\frac{7}{5}(Q_{T'})} \le C
\]
for the Navier--Stokes equations, and the inequality
\[
\|\pa_t c\|_{L^2(Q_{T'})} + \|\, {\rm div} \, (c \vv)\|_{L^2(Q_{T'})} + \|\Delta c\|_{L^2(Q_{T'})} \le C
\]
for the convection-diffusion equation. Hence we have, from \eqref{vestrong} and \eqref{cestrong}, that
\begin{align*}
        \int_{Q_{T'}} \big(\pa_t \vv+ (\vv \cdot \nabla) \vv   - {\rm div} \, \SSS(c,\DD\vv) \big) \cdot \ffpsi \dx\dt &= \int_{Q_{T'}} \fff \cdot \ffpsi \dx \dt && \text{for all }  \ffpsi \in C_0^\infty(I';\mathcal{V}), \\
        \int_{Q_{T'}} \big(\pa_t c  + {\rm div}\,  (c\vv)  - \Delta c  \big)\, \phi\dx\dt &= 0 && \text{for all }  \phi \in C_0^\infty(I';C^\infty(\Omega)).
\end{align*}

Now, we shall apply de Rahm's theorem to recover the kinematic pressure, which ensures that $\nabla \pi \in L^\frac{7}{5}(I' \times \Omega)$ (see, e.g., \cite{breit, Simon1993}). Therefore, we have that
\begin{alignat*}{2}
    \pa_t \vv + (\vv \cdot \nabla) \vv - {\rm div} \, \SSS(c,\DD\vv) &=- \nabla \pi + \fff \quad &&\text{a.e. in } Q_{T'},\\ 
    \pa_t c + {\rm div}\,(c\vv) - \Delta c &= 0 \quad &&\text{a.e. in }Q_{T'}. 
\end{alignat*}

Finally, to complete the proof of Theorem \ref{maintheorem}, it remains to identify the initial conditions as $\vv(0) = \vv_0$ and $c(0) = c_0$. As before, since $\vv, c \in L^2(I';W^{1,2}(\Omega))$ and $\pa_t \vv, \pa_t c \in L^2(I';L^2(\Omega))$, we see that $\vv, c \in C(I';L^2(\Omega))$ . For the identification of the initial velocity, we let $\delta \to 0+$ in \eqref{vdelta0} to get
\[
\int_\Omega \vv(\tau) \cdot \ww_i \dx + \int_0^\tau \int_\Omega (( \vv \cdot \nabla) \vv  - {\rm div}\, \SSS(c, \DD\vv) -\fff) \cdot \ww_i \dx \dt = \int_\Omega \vv_0 \cdot \ww_i \dx.
\]
Taking $\tau \to 0+$ with the regularity $\vv \in C(I';L^2(\Omega))$ implies $\vv(0) =\vv_0$ as desired. For the concentration, we shall take $\delta \to 0+ $ in \eqref{cdelta0} to obtain
\[
\int_\Omega c(\tau) \varphi \dx + \int_0^\tau \int_\Omega (\vv \cdot \nabla c - \Delta c) \varphi \dx\dt = \int_\Omega c_0 \varphi \dx,
\]
where we used $\|c_{0,\delta}- c_0\|_2\to0$. Therefore, by noting $c \in C(I';L^2(\Omega))$, taking the limit $\tau \to 0+$ yields $c(0) =c_0$, which completes the proof.

\begin{remark}
    As described in Corollary 18 in \cite{Diening2005}, the regularity of $\vv$ can be further improved. More precisely, we can deduce $(\overline{\DD}\vv)^\frac{p^-}{2} \in L^\frac{2(5p^--6)}{2-p^-}(I'; W^{1,2}(\Omega))$ and $\pa_t( (\overline{\DD}\vv)^\frac{p^-}{2} ) \in L^2(I';L^2(\Omega))$, which implies $(\overline{\DD}\vv)^\frac{p^-}{2} \in C(I',L^\frac{2m}{p^-}(\Omega))$ for $1 \le m < 6(p^--1)$. Thus $\vv \in C(I';W^{1,m}(\Omega))$ for $1 \le m < 6(p^--1)$.
\end{remark}

\section{Proof of Theorem \ref{uniqueness theorem}}\label{sec:unique}
In this section, we shall prove the uniqueness of the strong solution presented in Theorem \ref{uniqueness theorem}. Let $(\vv_1,c_1)$ and $(\vv_2,c_2)$ be two strong solutions of the problem \eqref{main_sys}-\eqref{concentration}, which were shown to exist in the previous section. First, if we take the difference between the equations \eqref{main_sys} corresponding to $\vv_1$ and $\vv_2$ and test $\vv_1 - \vv_2$, we have
\begin{equation} \label{uniPQ}
    \frac{1}{2} \frac{\rm d}{\dt} \|\vv_1 - \vv_2\|_2^2 + {\rm P} = {\rm Q},
\end{equation}
where
\begin{equation} \label{P and Q}
\begin{split}
    {\rm P} &:= \int_\Omega \left( \SSS(c_1,\DD\vv_1)-\SSS(c_2,\DD\vv_2)\right) : (\DD\vv_1-\DD\vv_2) \dx, \\
    {\rm Q} &:= - \int_\Omega \left( (\vv_1\cdot \nabla )\vv_1 - (\vv_2 \cdot \nabla ) \vv_2 \right) \cdot (\vv_1 - \vv_2) \dx.
\end{split}
\end{equation}
Let us write ${\rm P}$ as
\begin{equation} \label{uniN}
\begin{split}
    {\rm P} &= \int_\Omega \left( \SSS(c_1,\DD\vv_1)-\SSS(c_1,\DD\vv_2) \right):(\DD\vv_1 - \DD\vv_2)\dx \\
    & \hspace{4mm} + \int_{\Omega}\left(\SSS(c_1,\DD\vv_2) - \SSS(c_2,\DD\vv_2)\right) : (\DD\vv_1-\DD\vv_2) \dx \\
    &=: {\rm P_1} + {\rm P_2}.
\end{split}
\end{equation}
Due to the monotonicity \eqref{P2}, we have for the term $\rm P_1$ that
\begin{equation}\label{P111}
{\rm P_1} \ge C \int_\Omega (1+ |\DD\vv_1|^2 + |\DD\vv_2|^2)^\frac{p(c_1)-2}{2}|\DD\vv_1 - \DD\vv_2|^2 \dx .
\end{equation}
Moreover, by applying the mean value theorem, there exists $z$ between $c_1$ and $c_2$ such that
\[
\begin{split}
    {\rm P_2} 
    &= 2\nu_0\int_\Omega p'(z) (1+ |\DD\vv_2|)^{\frac{p(z)-2}{2}} \log(1+ |\DD\vv_2|^2) \DD\vv_2 : (\DD\vv_1 - \DD\vv_2) (c_1-c_2)\dx.
\end{split}
\]
Since $p$ is Lipschitz continuous and $\log(1+ |\DD\vv|^2) \le C (1+ |\DD\vv|^2)^\alpha$ for some small $\alpha>0$, we can estimate the second term $\rm P_2$ as
\begin{equation*} 
\begin{split}
    |{\rm P_2}| &\le C \int_\Omega (1+ |\DD\vv_2|^2)^{\frac{p(z)-2}{2} + \frac{1}{2} + \alpha} |\DD\vv_1- \DD\vv_2||c_1-c_2| \dx \\
    & \le C \int_\Omega (1+ |\DD\vv_1|^2 + |\DD\vv_2|^2)^{\frac{p(z)-1}{2} + \alpha} |\DD\vv_1- \DD\vv_2||c_1-c_2| \dx \\
    &= C\int_\Omega (1+ |\DD\vv_1|^2 + |\DD\vv_2|^2)^{\frac{2p(z)-p(c_1)}{4} + \alpha } (1+|\DD\vv_1|^2+|\DD\vv_2|^2)^\frac{p(c_1)-2}{4}|\DD\vv_1 - \DD\vv_2||c_1 - c_2| \dx,
\end{split}
\end{equation*}
where we have used $\frac{p(z)-1}{2} \ge 0$ for the second inequality above. Now by applying H\"older's inequality and Young's inequality, we have for sufficiently small $\epsilon>0$ that
\begin{equation} \label{P22}
\begin{aligned}
    |{\rm P_2}| &\le C\left( \int_\Omega (1+|\DD\vv_1|^2 +  |\DD\vv_2|^2)^{3(\frac{2p^+ - p^-}{4} + \alpha)} \dx \right)^\frac{2}{3} \|c_1- c_2\|_6^2 \\
    &\hspace{4mm} + \epsilon \int_\Omega (1+ |\DD\vv_1|^2 + |\DD\vv_2|^2)^\frac{p(c_1)-2}{2} |\DD\vv_1-\DD\vv_2|^2 \dx. 
\end{aligned}
\end{equation}
To further estimate ${\rm P_2}$, we need to obtain additional regularity of the velocity field. From \eqref{Ipv2}, \eqref{uniformest1} and \eqref{uniformest2}, we can derive 
\begin{equation}\label{Ipvhigh}
\int_{I'}\calI_p(c^n_\delta,\vv^n_\delta)^\frac{5p^--6}{2-p^-} \le C,
\end{equation}
where the above constant $C$ does not depend on $n\in\mathbb{N}$ and $\delta>0$. Moreover, by applying the Sobolev embedding, we have
\begin{equation} \label{highregularityDv}
\|\overline{\DD}\vv_\delta^n\|_{3p^-}^{p^-}= \|(\overline{\DD}\vv_\delta^n)^\frac{p^-}{2})\|_6^2 \le C\|\nabla ((\overline{\DD}\vv_\delta^n)^\frac{p^-}{2})\|_2^2  + C\|(\overline{\DD}\vv_\delta^n)^\frac{p^-}{2})\|_{2}^{2} \le C \calI_p(c_\delta^n,\vv_\delta^n) +C\|\overline{\DD}\vv^n_\delta\|_{p^-}^{p^-}.
\end{equation}
By combining \eqref{convergences}, \eqref{Ipvhigh}, and \eqref{highregularityDv}, and applying the limiting arguments $n\rightarrow\infty$ and $\delta\rightarrow0$ as before, we may conclude that
\begin{equation}\label{more_reg_Dv}
\DD\vv \in L^{\frac{5p^--6}{2-p^-}p^-}(I';L^{3p^-}(\Omega)).
\end{equation}
In order to bound the terms in \eqref{P22} using \eqref{more_reg_Dv}, the exponents should satisfy
\begin{equation} \label{p+p-}
\frac{3}{2}(2p^+-p^-) + 6\alpha \le 3p^- \quad\text{and}\quad 2p^+ - p^- + 4\alpha \le \frac{5p^--6}{2-p^-} p^-,
\end{equation}
regarding the spatial integrability and the temporal integrability, respectively. Note that $\alpha$ can be chosen sufficiently small so that the condition $p^+ < \frac{2p^--2}{2-p^-} p^-$ ensures the validity of \eqref{p+p-}. Consequently, applying the Sobolev embedding, we can rewrite \eqref{P22} as
\begin{equation}\label{P222}
|{\rm P_2}| \le F(t) \|\nabla (c_1-c_2)\|_2^2 + \epsilon \int_\Omega (1+ |\DD\vv_1|^2 + |\DD\vv_2|^2)^\frac{p(c_1)-2}{2} |\DD\vv_1-\DD\vv_2|^2 \dx, 
\end{equation}
for some $F(t) \in L^1(I')$. By combining the estimates \eqref{P111} and \eqref{P222} with \eqref{uniN}, we get
\[
C\int_\Omega (1+ |\DD\vv_1|^2 + |\DD\vv_2|^2)^\frac{p(c_1)-2}{2}|\DD\vv_1 - \DD\vv_2|^2 \dx \le F(t) \|\nabla(c_1-c_2)\|_2^2 + {\rm P}.
\]
Furthermore, since $p^->\frac{7}{5}$, there exists $\ell\in(\frac{8}{5},2)$ satisfying
\[
\frac{2-p(c_1)}{2} \cdot \frac{2\ell}{2-\ell} \le \frac{12}{5}.
\]
Therefore, from $\vv_1, \vv_2 \in L^\infty(I';W^{1,\frac{12}{5}}(\Omega))$, we obtain
\[
\|(\overline{\DD}\vv_1)^\frac{2-p(c)}{2} + (\overline{\DD}\vv_2)^\frac{2-p(c)}{2} \|_{\frac{2\ell}{2-\ell}} \in L^\infty(I'),
\]
which, together with Lemma \ref{difference_q}, implies
\begin{equation}\label{unique_P}
    C\|\DD\vv_1 - \DD\vv_2\|_\ell^2 \le F(t) \|\nabla(c_1-c_2)\|_2^2 + {\rm P}.
\end{equation}

Next, we derive an estimate for \( \mathrm{Q} \) appearing in \eqref{P and Q}.  By using H\"older's inequality with the regularity $\vv_1 \in L^\infty(I';W^{1,\frac{12}{5}}(\Omega))$, we may write
\[
{\rm Q} = - \int_\Omega ((\vv_1-\vv_2)\cdot \nabla ) \vv_1 \cdot(\vv_1 - \vv_2)\dx  \le \|\vv_1 - \vv_2\|_\frac{24}{7}^2 \|\nabla \vv_1\|_\frac{12}{5} \le C \|\vv_1 - \vv_2\|_\frac{24}{7}^2.
\]
Since $\ell\in(\frac{8}{5},2)$ implies $2< \frac{24}{7}<\frac{3\ell}{3-\ell}$, the application of Gagliardo--Nirenberg inequality, Korn's inequality and Young's inequality yields
\begin{equation}\label{unique_Q}
    {\rm Q} \le C \|\vv_1 - \vv_2\|_2^{2\theta} \|\DD\vv_1 - \DD\vv_2 \|_\ell^{2(1-\theta)} \le C\|\vv_1 - \vv_2\|_2^2 + \epsilon\|\DD\vv_1 - \DD\vv_2\|_\ell^2,
\end{equation}
where $\frac{7}{24}= \frac{\theta}{2} + \frac{3-\ell}{3\ell}(1-\theta)$.
Hence \eqref{uniPQ} together with \eqref{unique_P} and \eqref{unique_Q} implies
\begin{equation}\label{unique_v_final}
\frac{1}{2} \frac{\rm d}{\dt} \|\vv_1 - \vv_2\|_2^2 + C \|\DD\vv_1 - \DD\vv_2\|_\ell^2 \le F(t) \|\nabla (c_1 - c_2)\|_2^2 + C\|\vv_1 - \vv_2 \|_2^2.
\end{equation}

As a second step, for the convection-diffusion equation, we take the difference of the equation \eqref{concentration} corresponding to $c_1$ and $c_2$ to get
\begin{equation} \label{c1c2equation}
    \pa_t(c_1-c_2) - \Delta(c_1-c_2) + {\rm div}\, (c_1\vv_1 - c_2 \vv_2) = 0.
\end{equation}
We then test $-\Delta (c_1-c_2)$ to \eqref{c1c2equation} to obtain
\begin{equation}\label{uniconcentration}
\begin{aligned}
   \frac{1}{2}\frac{\rm d}{\dt} \int_\Omega &|\nabla (c_1-c_2)|^2 \dx + \int_\Omega |\nabla^2 (c_1-c_2)|^2 \dx\\ 
    &\le C\int_\Omega|\vv_1 \cdot \nabla c_1 - \vv_2 \cdot \nabla c_2|^2 \dx= C\int_\Omega |(\vv_1-\vv_2) \cdot \nabla c_1  + \vv_2 \cdot (\nabla c_1 - \nabla c_2)|^2\dx\\
    &\le C \int_\Omega |(\vv_1 - \vv_2) \cdot \nabla c_1|^2\dx + C\int_{\Omega}|\vv_2 \cdot\nabla(c_1-c_2)|^2 dx \\
    &=: {\rm U_1} + {\rm U_2}.
\end{aligned}
\end{equation}
For the term ${\rm U_1}$, we first note that for any solution pair $(\vv,c)$ constructed in the previous section, $\vv \cdot \nabla c$ belongs to $ L^\infty(I';L^3(\Omega))$, since $\vv \in L^\infty(I',W^{1,\frac{12}{5}}(\Omega))\hookrightarrow L^\infty(I',L^{12}(\Omega))$ and $\nabla c \in L^\infty(I';L^4(\Omega))$. Moreover, since $c_0 \in W^{2,q_0}(\Omega)$ for some $q_0\geq4$, we can apply the maximal regularity theory (see, e.g., Theorem 5.4 and Theorem D.12 of \cite{james}) to \eqref{concentration}, which yields $\Delta c \in L^\nu(I';L^3(\Omega))$ for any $1\le \nu <\infty$. Therefore, from the standard Calder\'on--Zygmund estimate on a torus (see,
e.g., Theorem B.7 in \cite{james}), it follows that $\nabla c \in L^\nu(Q_{T'})$ for any $1\le \nu <\infty$. Now we shall use H\"older's inequality to get
\begin{equation*} 
{\rm U_1} \le \|\vv_1 - \vv_2\|_\frac{24}{7}^2 \|\nabla c_1\|_\frac{24}{5}^2.
\end{equation*}
Furthermore, by using the Gagliardo--Nirenberg inequality with $2<\frac{24}{7}<\frac{3\ell}{3-\ell}$, Korn's inequality and Young's inequality, we obtain 
\begin{align*}
{\rm U_1} &\le C \|\nabla c_1\|_\frac{24}{5}^2 \|\vv_1 - \vv_2\|_2^{2\theta} \|\DD\vv_1 - \DD\vv_2\|_\ell^{2(1-\theta)}\\
&\le C\|\nabla c_1\|_\frac{24}{5}^\frac{2}{\theta} \|\vv_1- \vv_2\|_2^2 + \epsilon\|\DD\vv_1 - \DD\vv_2\|_\ell^2,
\end{align*}
where $\frac{7}{24} = \frac{\theta}{2} + \frac{3-\ell}{3\ell}(1-\theta)$.
Since $\nabla c_1 \in L^\nu(Q_{T'})$ for any $1 \le \nu < \infty$, we conclude that
\[
{\rm U_1} \le G(t) \|\vv_1- \vv_2\|_2^2 + \epsilon\|\DD\vv_1 - \DD\vv_2\|_\ell^2,
\]
for some $G(t) \in L^1(T')$. For the second term $\rm U_2$, by H\"older's inequality and the Sobolev embedding, there holds
\[
{\rm U_2} \le \|\vv_2\|_{12}^2 \|\nabla c_1 - \nabla c_2\|_\frac{12}{5}^2 \le \|\nabla \vv_2\|_\frac{12}{5}^2\|\nabla c_1- \nabla c_2\|_\frac{12}{5}^2 \le C\|\nabla c_1- \nabla c_2\|_\frac{12}{5}^2,
\]
where we have used $\vv \in L^\infty(I';W^{1,\frac{12}{5}}(\Omega))$ for the last inequality. Thus, by the Gagliardo--Nirenberg inequality and Young's inequality, we obtain
\[
{\rm U_2} \le C\|\nabla(c_1- c_2)\|_2^2 + \epsilon \|\nabla^2(c_1-c_2)\|_2^2.
\]
By collecting \eqref{unique_v_final} and \eqref{uniconcentration} with the above estimates for ${\rm U_1}$ and ${\rm U_2}$, we obtain 
\begin{align*}
 \frac{\rm d}{\dt} &\|\vv_1 - \vv_2\|_2^2 +  \frac{\rm d}{\dt} \|\nabla(c_1-c_2)\|_2^2 +C \|\nabla^2(c_1-c_2)\|_2^2 + C \|\DD\vv_1 - \DD\vv_2\|_\ell^2 \\
&\le \widetilde{F}(t) \|\nabla (c_1 - c_2)\|_2^2 + \widetilde{G}(t)\|\vv_1 - \vv_2 \|_2^2,
\end{align*}
for some $\widetilde{F}(t), \widetilde{G}(t) \in L^1(I')$. Finally, if we apply Gr\"onwall's inequality to the above inequality, we get
\[
\vv_1-\vv_2 =0, \quad \nabla(c_1-c_2) =0.
\]
For the concentrations, the Poincar\'e inequality implies that $c_1-c_2=0$. This completes the proof of the uniqueness theorem.

\section{Two-dimensional problem}\label{sec:2D}

\subsection{Main theorems in two dimensions}
While the primary objective of this paper is to establish the local-in-time existence and uniqueness of strong solutions on the three-dimensional torus, the analytical framework developed herein is readily adaptable, in a simplified form, to the two-dimensional setting. In the latter case, the proof can be carried out under less restrictive conditions on the exponent. The key difference lies in the fact that, in the two-dimensional case, the term $\|\nabla \vv\|^3_3$ does not appear in \eqref{deltavtest} due to the combinatorial symmetry of the partial derivatives. In the proof for the three-dimensional case, considerable effort was devoted to handling the term $\|\DD\vv\|_{\frac{12}{5}}$, which arises from the term $\|\nabla \vv\|_3^3$. However, its absence in the two-dimensional setting eliminates the need for such treatment, thereby allowing significant simplification of the overall proof. For the sake of completeness, we outline the proof in the two-dimensional case in this section.

Now, let us consider the system of equations \eqref{main_sys}-\eqref{concentration} in a two-dimensional periodic domain $\Omega = [0,1]^2$. The main result concerning the local well-posedness in the two-dimensional domain is encapsulated in the following theorem. Note here that, in this case, the existence and uniqueness can be established within the whole shear-thinning regime, and the integrability condition on $c_0$ can also be relaxed.\begin{theorem} \label{maintheorem2d}
Suppose that $p: \R \to \R$ is a Lipschitz continuous function with $1 < p^- \leq p(\cdot)\leq p^+ \le 2$. Assume further that
\[
\|\vv_0\|_{2,2} + \|c_0\|_{2,q_0} + \|\fff\|_{L^\infty(I;W^{1,2}(\Omega))} + \|\pa_t \fff\|_{L^2(I;L^2(\Omega))}  \le C
\]
with $q_0>2$.
Then, there exists $T'>0$ with $0<T'<T$ such that the system \eqref{main_sys}-\eqref{concentration} admits a strong solution $(\vv,\pi, c)$ on the interval $I'=[0,T']$ satisfying 
\begin{align*}
    & \vv \in C(I';L^2(\Omega)) \cap L^\infty(I';W_{\rm div}^{1,2}(\Omega)) \cap L^{p^-}(I'; W^{2,p^-}(\Omega)),\\
    & c \in C(I';L^2(\Omega)) \cap L^\infty(I';W^{1,2}(\Omega)) \cap L^2(I';W^{2,2}(\Omega)),
    \\
    &\partial_t \vv \in L^\infty(I';L^2(\Omega)), \quad \pa_t c \in L^\infty(I';L^2(\Omega)), \quad \nabla\pi\in L^{\min\{p^-,\frac{2q_0}{2+q_0}\}}(\Omega).
\end{align*}
\end{theorem}
In contrast to the three-dimensional case, the uniqueness theorem holds as follows without any additional assumptions on the exponent.

\begin{theorem} \label{uniqueness theorem2d}
    Under the assumptions of Theorem \ref{maintheorem2d}, the strong solution constructed in Theorem \ref{maintheorem2d} is unique. 
\end{theorem}

\subsection{Construction of Galerkin system and a priori estimates}
In the rest of this section, we aim to derive some {\textit{a priori}} estimates which will be used to prove Theorem \ref{maintheorem2d} and Theorem \ref{uniqueness theorem2d}. To do that, as in the three-dimensional case, we consider a regularized system
\begin{equation*} 
\begin{aligned}
\pa_t\vv_\delta + {\rm div}(\vv_\delta \otimes \vv_\delta) - {\rm div}\SSS(c_\delta, \DD \vv_\delta)  &= - \nabla \pi_\delta  + \fff & \text{in } Q_T,  \\
{\rm div}\, \vv_\delta &= 0 &  \text{in } Q_T, \\ 
\pa_t c_\delta + {\rm div}(c_\delta \vv_\delta) -\Delta c_\delta & = 0 & \text{in } Q_T,
\end{aligned}
\end{equation*}
with the initial conditions $\vv_\delta(\cdot,0) =\vv_0$ and $c_\delta(\cdot,0) = c_{0,\delta} := \eta_\delta * c_0$ where $\eta_\delta$ is standard molifier with $\delta>0$. Furthermore, we also consider the Galerkin system
\begin{align*}
\innerproduct{\pa_t \vv_\delta^{n,m}}{\ww_i} + \innerproduct{(\vv_\delta^{n,m}\cdot \nabla) \vv_\delta^{n,m}}{\ww_i} + \innerproduct {\SSS_\delta^{n,m}}{\DD \ww_i} &= \innerproduct{\fff}{\ww_i},  \\ 
\innerproduct{\pa_tc_\delta^{n,m}}{z_j} + \innerproduct{\vv_\delta^{n,m} \cdot \nabla c_\delta^{n,m}}{z_j} + \innerproduct{\nabla c_\delta^{n,m}}{\nabla z_j} &=0, 
\end{align*}
for all $i=1,\ldots n$ and $j=1,\ldots m$. Here, we write
\[
\vv_\delta^{n,m} = \sum_{i=1}^n \alpha_i^{n,m}(t)\ww_i \quad\text{and}\quad c_\delta^{n,m} = \sum_{j=1}^m \beta_j^{n,m}(t)z_j,
\]
with $\ww_i$ forming a basis of $W^{s,2}_{\rm div}(\Omega)$, $s>2$ and $z_j$ forming a basis for $W^{1,2}(\Omega)$. The initial conditions can be written as
\begin{equation*}
    \vv_\delta^{n,m}(\cdot,0) = P^n\vv_0 \quad\text{and}\quad c_\delta^{n,m}(\cdot,0) = P^mc_{0,\delta},
\end{equation*}
where $P^n$, $P^m$ are the orthogonal projection onto $A_n:={\rm span}\{\ww_1,\ldots,\ww_n\}$ and $B_m:={\rm span}\{z_1,\ldots,z_m\}$, respectively. The existence of the approximate system again follows from Carath\'eodory's theorem and the relevant uniform estimates. Similar to Section \ref{sec:galerkin}, we first let $m \to \infty$ for the convection-diffusion equation. Then, for all $i=1,\ldots,n$, there holds
\begin{align}
\innerproduct{\pa_t \vv_\delta^n}{\ww_i} + \innerproduct{(\vv_\delta^n\cdot \nabla) \vv_\delta^n}{\ww_i} + \innerproduct {\SSS_\delta^n}{\DD \ww_i} &= \innerproduct{\fff}{\ww_i}, \label{aux2d1}\\ 
\innerproduct{\pa_tc_\delta^n}{\varphi} + \innerproduct{\vv_\delta^n \cdot \nabla c_\delta^n}{\varphi} + \innerproduct{\nabla c_\delta^n}{\nabla \varphi} &= 0, \label{aux2d2}
\end{align}
for arbitrary $\varphi \in W^{1,2}(\Omega)$ with the initial conditions
\begin{equation*}
    \vv_\delta^n(\cdot,0) = P^n\vv_0 \quad\text{and}\quad c_\delta^n(\cdot,0) = c_{0,\delta}.
\end{equation*}
Moreover, as before, $c^n_\delta$ satisfies
\begin{equation*} 
c_\delta^n, \, \pa_t c_\delta^n, \,
\pa_t^2 c_\delta^n \in L^\infty(I;C^\infty(\Omega)),
\end{equation*}
and the equation \eqref{aux2d2} can be expressed in a pointwise manner:
\begin{equation}\label{regularequationcn2d}
    \pa_t c_\delta^n + \vv_\delta^n \cdot \nabla c_\delta^n - \Delta c_\delta^n = 0 \quad \text{a.e. in } Q_T.
\end{equation}

Before we derive uniform estimates, we begin by considering the lemmas presented in Section \ref{sec:prelim}. To begin with, we observe that the result of Lemma \ref{DDv,Dv} and Lemma \ref{difference_q} still hold in the two-dimensional setting, as can be seen from the proofs in \cite{Diening2005}. We then derive {\textit{a priori}} estimates analogous to those in Section \ref{uniform est 3D}. Again, for the sake of simplicity, we will omit the indices $n\in\mathbb{N}$ and $\delta>0$. We start with testing $\vv$ to \eqref{aux2d1} to get
\begin{equation} \label{best2d}
    \sup_{0\le t \le T} \norm{\vv(t)}_2^2 + \int_0^T \int_{\Omega} \left(|\nabla \vv|^{p(c)} + |\SSS(c,\DD\vv)|^{(p(c))'} \right)\dx\dt \le C.
\end{equation}
Next, we shall test $-\Delta \vv$ to \eqref{aux2d1}, and we may write
\begin{equation}\label{2D_v_int_est}
\int_{\Omega} \pa_t \vv \cdot (-\Delta \vv) \dx - \int_{\Omega} ((\vv \cdot \nabla) \vv) \cdot (-\Delta \vv) \dx + \int_{\Omega} \SSS(c,\DD \vv):\DD (-\Delta \vv) \dx = \int_{\Omega} \fff \cdot (-\Delta \vv) \dx.
\end{equation}
As we mentioned at the beginning of this section, if we integrate the convection term in \eqref{2D_v_int_est} by parts, we can see that it vanishes in two dimensions and the term $\|\nabla\vv\|_3^3$ in \eqref{deltavtest} does not appear, which allows us to avoid handling the term $\|\DD\vv\|_{\frac{12}{5}}$.
Therefore, from \eqref{2D_v_int_est}, we may write
\begin{equation} \label{Ipv2d}
\frac{1}{2} \frac{\rm d}{\dt} \|\nabla \vv\|_2^2 + C\calI_p(c,\vv) \le |{\rm A}| + \int_\Omega |\nabla \fff \cdot \nabla \vv| \dx ,
\end{equation}
where 
\[
{\rm A} :=\int_{\Omega}  \pa_{x_k} c  \left(\frac{\pa\SSS(c,\DD)}{\pa c} : \pa_{x_k}\DD \vv\right)
  \dx.
\]
Furthermore, let us next take the time derivative to \eqref{aux2d1} and test $\pa_t \vv$. Then we have
\begin{equation} \label{Jpv2d}
\frac{1}{2} \frac{\rm d}{\dt} \|\pa_t \vv\|_2^2 + C \calJ_p(c,\vv) \le |{\rm B}| + \int_\Omega |\pa_t \fff\cdot  \pa_t \vv | \dx + \int_\Omega |\pa_t((\vv \cdot \nabla) \vv) \cdot \pa_t \vv| \dx,
\end{equation}
where
\[
{\rm B}:= \int_\Omega \pa_t c \left(\frac{\pa \SSS(c,\DD)}{\pa c} : (\DD \pa_t \vv)\right)  \dx.
\]
Note that, since the terms $\|\nabla \vv\|_3$, $\|\DD\vv\|_\frac{12}{5}$, and $\int_\Omega \pa_t \vv \cdot \Delta \vv  \mathrm{d}x$ do not appear in the two-dimensional setting, Lemma \ref{Dvss}, Lemma \ref{patvdeltav}, and the first part of Lemma \ref{gradv3} are no longer required. 
We shall modify the second inequality in Lemma~\ref{gradv3} to better suit the two-dimensional setting. To do that, we first introduce the following lemma concerning the estimates for the time derivative of the velocity gradient.
\begin{lemma}[Lemma 7 in \cite{Diening2005}] \label{DDv Dpatv}
    Let $\calJ_p(c,\vv)$ be the energy defined in \eqref{special_energy} and $1\le r \le 2$. Then for sufficiently smooth $\vv$, we have 
    \begin{equation*}
        \|\pa_t\DD  \vv\|_r \le C\calJ_p(c,\vv)^\frac{1}{2} \|\overline{\DD}\vv^\frac{2-p^-}{2}\|_\frac{2r}{2-r}.
    \end{equation*}
\end{lemma}
We then state the modified version of the second statement of Lemma \ref{gradv3} as follows.
\begin{lemma}
    There exists constants $R_1, R_2 >2$, such that for sufficiently small $\epsilon>0$, there holds
    \begin{equation} \label{patv gradv 2d}
\int_\Omega \pa_t((\vv \cdot \nabla) \vv) \cdot \pa_t \vv \dx \le \epsilon \calJ_p(c,\vv) + C \left(\|\pa_t \vv\|_2^{R_{1}} + \|\nabla \vv\|_2^{R_{2}} +1 \right).
\end{equation}
\end{lemma}

\begin{proof}
   By applying H\"older's inequality, Korn's inequality and the Gagliardo--Nirenberg interpolation inequality, we have
   \[
   |\innerproduct{\pa_t \vv \cdot \nabla \vv}{\pa_t \vv}| \le \|\pa_t \vv\|_4^2 \|\nabla \vv\|_2 \le C\|\pa_t\vv\|_2^{2\theta} \|\pa_t \DD\vv\|_\frac{4}{4-p^-}^{2(1-\theta)} \|\nabla \vv\|_2,
   \]
   for some $\theta\in(0,1)$. Then Lemma \ref{DDv Dpatv} implies
   \[
   |\innerproduct{\pa_t \vv \cdot \nabla \vv}{\pa_t \vv}| \le C \|\pa_t \vv\|_2^{2\theta} \|\nabla \vv\|_2 \calJ_p(c,\vv)^{1-\theta} (1+ \|\DD\vv\|_2)^{(2-p^-)(1-\theta)}.
   \]
   Hence, it follows from Young's inequality that
   \[
   |\innerproduct{\pa_t \vv \cdot \nabla \vv}{\pa_t \vv}| \le \epsilon \calJ_p(c,\vv) + \|\pa_t \vv\|_2^2\|\nabla \vv\|_2^\frac{1}{\theta}( 1+ \|\nabla \vv\|_2)^{(2-p^-) \frac{1-\theta}{\theta}}.
   \]
   By applying Young's inequality once more, we obtain the desired result.
\end{proof}
Next, we shall derive estimates for ${\rm A}$ and ${\rm B}$ in two dimensions by slightly modifying the proof of Lemma \ref{A2B2}, replacing the term $\|\overline{\DD}\vv\|_{\frac{12}{5}}$ with $\|\nabla \vv\|_2$. 
\begin{lemma}
    Let $q>2$. Then there exist constants $R_3,R_4,R_5>q$ and $R_6,R_7 > 2$ such that
    \begin{align}
    &|{\rm{A}}| \le C\|\nabla c\|_q^{R_3} + C \|\nabla \vv\|_2^{R_6} + \epsilon \calI_p(c,\vv) + C, \label{A22d} \\
    &|{\rm{B}}| \le C\|\pa_t c\|_q^{R_4} + C\|\nabla c\|_q^{R_5} +  C \|\nabla\vv\|_2^{R_7} + \epsilon \calI_p(c,\vv) + \epsilon \calJ_p(c,\vv) + C. \label{B22d}
    \end{align}
\end{lemma}
\begin{proof}
    For the first term $\rm A$,  as we did in the proof of Lemma~\ref{A2B2}, we proceed as
    \[
    |{\rm A}| \le C \|\nabla c\|_q^2\|\log(2+ |\DD\vv|)\|_\alpha^2 \|\overline{\DD}\vv\|_\beta^2 + \epsilon\calI_p(c,\vv),
    \]
    where $\frac{1}{2} = \frac{1}{q} + \frac{1}{\alpha} + \frac{1}{\beta}$ with large $\alpha>2$ and $\beta>1$ being determined later. As before, we define $\eta:= \overline{\DD}\vv^\frac{p(c)}{2}$ and apply the Gagliardo--Nirenberg inequality to get
    \[
    \|\eta\|_\beta^2 \le C\|\eta\|_2^{2\theta}\|\nabla \eta\|_2^{2(1-\theta)} + C\|\eta\|_2^2,
    \]
    where $\theta = \frac{2}{\beta}$ and $\beta \in (2,\infty)$ is chosen to satisfy $\frac{1}{2} = \frac{1}{q} + \frac{1}{\alpha} + \frac{1}{\beta}$. We follow the procedure for the three-dimensional proof with the estimate $\|\eta\|_2^2 \le \|\overline{\DD}\vv\|_2^2 \le 1+ \|\nabla \vv\|_2^2$ and $\log(2+ |\DD\vv|) \le C(1+ |\DD\vv|)^\frac{2}{\alpha} \le C (1+ |\nabla \vv|)^\frac{2}{\alpha}$ to conclude \eqref{A22d}. The result for $\rm B$ can be obtained in a straightforward manner, similarly to the proof of the estimate for $\rm B_2$ in Lemma~\ref{A2B2}.
\end{proof}

Finally, we need to handle the terms $\nabla c$ and $\pa_t c$ as we did in Section \ref{uniform est 3D}.  A slight modification of the proofs of Lemmas \ref{gradcqlemma} and Lemma \ref{patclemma} yields the desired estimates for $\|\nabla c\|_q$ and $\|\pa_t c\|_q$ in two dimensions, which are summarized in the following lemmas.

\begin{lemma} 
    Let $q \ge 2$. Then there exist constants $R_{8}>2$ and $R_{9}>q$ such that
    \begin{equation} \label{nablacqq2d} 
        \frac{\rm d}{\dt} \|\nabla c \|_q^q + C\int_\Omega \left|\nabla (|\nabla c |^\frac{q}{2}) \right|^2 \dx \le C \|\nabla\vv\|_2^{R_{8}} + C \|\nabla c \|_q^{R_{9}} + C.
    \end{equation}
\end{lemma}
\begin{proof}
    As we did in the proof of Lemma \ref{gradcqlemma}, if we multiply $- \nabla \cdot (|\nabla c|^{q-2} \nabla c)$ to \eqref{regularequationcn2d}, we have
    \[
    \frac{\rm d}{{\rm d}t} \int_\Omega |\nabla c |^q \dx + C\int_\Omega \left|\nabla (|\nabla c |^\frac{q}{2}) \right|^2 \dx \le C \int_\Omega |\nabla \vv| |\nabla c|^q \dx .
    \]
    By applying Hölder’s inequality, we estimate the term on the right-hand side as
    \begin{equation*}
    \int_\Omega |\nabla \vv| |\nabla c |^q \dx \le \|\nabla \vv\|_2 \| |\nabla c |^\frac{q}{2} \|_4^2.
    \end{equation*}
    Furthermore, we use the Gagliardo--Nirenberg interpolation inequality to obtain
    \[
    \||\nabla c |^\frac{q}{2}\|_{4}^2 \le C\| |\nabla c|^\frac{q}{2}\|_2^2 \|\nabla (|\nabla c|^\frac{q}{2})\|_2^2 + C\||\nabla c|^\frac{q}{2}\|_2^2.
    \]
    Finally, the application of Young's inequality leads to the desired result.
\end{proof}

\begin{lemma}
    Let $q>2$. Then there exist constants $R_{10}>2$ and $R_{11},R_{12}>q$ such that
    \begin{equation} \label{patcqq2d} 
    \frac{\rm d}{\dt} \|\pa_t c\|_q^q + C\int_\Omega \left| \nabla (|\pa_t c|^\frac{q}{2}) \right|^2 \dx \le C \|\pa_t \vv\|_2^{R_{10}} + C\|\nabla c\|_q^{R_{11}} + C\|\pa_t c\|_q^{R_{12}} +C.
    \end{equation}
\end{lemma}
\begin{proof}
    Analogously to the proof of Lemma \ref{patclemma}, we have
    \[
    \frac{\rm d}{\dt} \|\pa_t c\|_q^q + C \int_\Omega \left| \nabla(|\pa_t c |^\frac{q}{2}) \right|^2 \dx \le C \int_\Omega |\pa_t \vv| |\nabla c| |\pa_t c|^{q-1}\dx.
    \]
    By applying H\"older's inequality and the Sobolev embedding, we have
    \begin{align*}
        \int_\Omega |\pa_t \vv| |\nabla c| |\pa_t c|^{q-1}\dx
        &\le \|\pa_t \vv\|_2\|\nabla c\|_q\||\pa_t c|^\frac{q}{2}\|_{\frac{4q-4}{q-2}}^{\frac{2q-2}{q}}\\
        &\le C \|\pa_t \vv\|_2 \|\nabla c \|_q \left( \||\pa_t c|^\frac{q}{2}\|_2^{\frac{2(q-1)}{q}} +\|\nabla (|\pa_t c|^\frac{q}{2})\|_2^\frac{2(q-1)}{q}\right).
    \end{align*}
    Since $\frac{2(q-1)}{q} <2$, by using Young's inequality, we can obtain the desired result.
\end{proof} 

We now use the above lemmas with $q=q_0$, where $q_0$ is the exponent specified in Theorem \ref{maintheorem2d}. In addition, the terms involving the forcing term in \eqref{Ipv2d} and \eqref{Jpv2d} can be estimated by H\"older's inequality and Young's inequality as 
\begin{align*}
    \int_\Omega |\nabla\fff \cdot \nabla \vv| \dx &\le \|\nabla \fff\|_2\|\nabla \vv\|_2 \le C\|\nabla \fff\|_2^2 +  C\|\nabla \vv\|^2_2 \le C + C \|\nabla\vv\|_2^2, \\
    \int_\Omega |\pa_t \fff\cdot  \pa_t \vv | \dx &\le \|\pa_t \fff\|_2\|\pa_t \vv\|_2 \le C \|\pa_t \fff\|_2^2 + C\|\pa_t \vv\|^2_2.
\end{align*}
Then, by collecting the estimates \eqref{Ipv2d}-\eqref{patcqq2d}, we deduce, for some exponents $S_1,S_2>2$ and $S_3,S_4>q$, that
\[
\begin{split}
     \frac{\rm d}{\dt} \|&\nabla\vv_\delta^n\|_2^2 + \frac{\rm d }{\dt}  \|\pa_t \vv_\delta^n\|_2^2 + \frac{\rm d}{\dt}\|\nabla c_\delta^n\|_q^q + \frac{\rm d}{\dt}\|\pa_t c_\delta^n\|_q^q  + C\calI_p(c_\delta^n,\vv_\delta^n) + C \calJ_p(c_\delta^n,\vv_\delta^n) \\
    & \le C\Big( 1 + \|\pa_t \fff\|_2^2+ \|\nabla\vv_\delta^n\|_2^{S_1} + \|\pa_t \vv_\delta^n\|_2^{S_2}+ \|\nabla c_\delta^n\|_q^{S_3} + \|\pa_t c_\delta^n\|_q^{S_4} \Big).
\end{split}
\]
To apply Lemma~\ref{Gronwalllemma}, we need to verify two conditions: the absolute continuity of $\|\nabla c^n_\delta\|_q$ and $\|\partial_t c^n_\delta\|_q$, and the boundedness of the initial data. Both conditions can be readily checked by the arguments analogous to those used in the three-dimensional case, and thus, we omit the details. In view of the assumption $\partial_t \boldsymbol{f} \in L^2(Q_T)$ from Theorem~\ref{maintheorem2d}, we may then apply the local Gr\"onwall inequality to obtain the following local uniform estimates: there exists a constant $C>0$ independent of $n\in\mathbb{N}$ and $\delta>0$ such that
\begin{equation}\label{uni_est_1}
\begin{aligned}
\|\nabla \vv^n_\delta\|_{L^\infty(I';L^2(\Omega))} &+ 
\|\pa_t \vv^n_\delta\|_{L^\infty(I';L^2(\Omega))} + \|\nabla c^n_\delta\|_{L^\infty(I';L^q(\Omega))} + \|\pa_t c^n_\delta\|_{L^\infty(I';L^q(\Omega))} \\
&+\|\calI_p(c^n_\delta,\vv^n_\delta)\|_{L^1(I')} + \|\calJ_p(c^n_\delta,\vv^n_\delta)\|_{L^1(I')} \le C,
\end{aligned}
\end{equation}
where $I'=[0,T']$ for some $T' \in (0,T)$.

\subsection{Proof of Theorem \ref{maintheorem2d}}
In this section, we shall prove the existence of the strong solutions of \eqref{main_sys}-\eqref{concentration} in two dimensions. Since we use an approach similar to the one employed in Section \ref{sec:limit}, we shall only outline the proof.
As a first step, from the uniform estimate \eqref{uni_est_1} with \eqref{grad^2vp}, we can extract a (not relabeled) subsequence with respect to $n\in\mathbb{N}$ such that
\begin{equation} \label{convergences 2d}
    \begin{aligned}
        \vv_\delta^n &\rightharpoonup \vv_\delta &&\mbox{weakly in }L^{p^-}(I';W^{2,p^-}(\Omega)), \\
     \vv_\delta^n & \overset{\ast} {\rightharpoonup} \vv_\delta && \text{weakly-$^*$ in } L^\infty(I';W^{1,2}(\Omega)), \\
     \partial_t \vv_\delta^n & \overset{\ast} {\rightharpoonup} \partial_t \vv_\delta && \text{weakly-$^*$ in } L^\infty(I';L^2(\Omega)), \\
     c_\delta^n & \overset{\ast}  {\rightharpoonup} c_\delta && \text{weakly-$^*$ in } L^\infty(I';W^{1,q}(\Omega)), \\
     \pa_t c_\delta^n & \overset{\ast} {\rightharpoonup} \pa_t c_\delta && \text{weakly-$^*$ in }L^\infty(I';L^q(\Omega)). 
    \end{aligned}
\end{equation}
Since $W^{2,p^-}(\Omega)$ is compactly embedded into $W^{1,2}(\Omega)$ for $p^->1$, by the Aubin--Lions lemma, we have  
\begin{equation*}
    \vv_\delta^n \to \vv_\delta \quad \mbox{strongly in }
    L^2(I';W^{1,2}(\Omega)),
\end{equation*} 
from which we obtain the pointwise convergence
\[
\DD\vv^n_\delta \to \DD\vv_\delta \quad \text{a.e. in }Q_{T'},
\]
and the strong convergence of the convection term
\[
(\vv_\delta^n \cdot \nabla) \vv_\delta^n \to (\vv_\delta\cdot \nabla) \vv_\delta \quad \text{in }L^\frac{4}{3}(Q_{T'}).
\]
Furthermore, for the approximate concentration, we have that
\[
c_\delta^n \to c_\delta \quad \text{strongly in } L^q(Q_{T'}),
\]
and the pointwise convergence
\begin{equation*}
    c_\delta^n \to c_\delta \quad \text{a.e. in } Q_{T'}.
\end{equation*}
As a consequence, thanks to the continuity of the stress tensor $\SSS$, we deduce
\begin{equation*}
    \SSS(c_\delta^n,\DD\vv_\delta^n) \to \SSS(c_\delta,\DD\vv_\delta) \quad \text{a.e. in }Q_{T'},
\end{equation*}
which, together with \eqref{best2d} and the Vitali convergence theorem, implies that
\[
\SSS(c^n_\delta, \DD\vv^n_\delta) \to \SSS(c_\delta, \DD\vv_\delta) \quad \text{in }L^1(Q_{T'}).
\]
Therefore, by passing $n\in\mathbb{N}$ to the limit, we have for any $\xi \in C^\infty_0(I')$ that 
\begin{equation*}
    \int_{I'} \left(\innerproduct{\pa_t \vv_\delta + (\vv_\delta \cdot \nabla) \vv_\delta}{\ww_i} + \innerproduct{ \SSS(c_\delta,\DD\vv_\delta)}{\DD\ww_i} \right) \xi \dt
    = \int_{I'} \big(\innerproduct{\fff}{\ww_i}\big) \xi  \dt.
\end{equation*}
Now, for $a= \min\{p^-,\frac{2q_0}{2+q_0}\}$, we have
\[
|\nabla \SSS(c_\delta,\DD\vv_\delta)|^a 
\le C|\nabla \DD\vv_\delta|^a + C(1+|\DD\vv_\delta|^2)^\frac{a}{2} \log^a(2+ |\DD\vv_\delta|)|\nabla c_\delta|^a.
\]
Since $q>2$, by Young's inequality, we see that
\begin{align*}
|\nabla \SSS(c_\delta,\DD\vv_\delta)|^a &\le C |\nabla \DD\vv_\delta|^a + C|\DD\vv_\delta|^2 + C\log^\alpha(2+ |\DD\vv_\delta|) + C|\nabla c_\delta |^q + C \\
&\le C |\nabla \DD\vv_\delta|^a + C|\DD\vv_\delta|^2 + C|\nabla c_\delta|^q + C,
\end{align*}
where $\frac{a}{2} + \frac{a}{\alpha} + \frac{a}{q} =1$ with $\alpha>1$ large enough. This, together with \eqref{convergences 2d}, implies $\nabla \SSS(c_\delta,\DD\vv_\delta) \in L^{\min\{p^-,\frac{2q_0}{2+q_0}\}}(Q_{T'})$.
Therefore, the regularized velocity field $\vv_\delta$ satisfies the uniform estimate
\[
\|\pa_t\vv_\delta\|_{L^\infty(I';L^2(\Omega))} + \|\nabla \SSS(c_\delta,\DD\vv_\delta)\|_{L^{\min\{p^-,\frac{2q_0}{2+q_0}\}}(Q_{T'})} + \|(\vv_\delta \cdot \nabla) \vv_\delta\|_{L^\frac{4}{3}(Q_{T'})} \le C,
\]
and the variational formulation
\begin{equation*} 
    \int_{Q_{T'}} \big(\pa_t \vv_\delta+ (\vv_\delta \cdot \nabla) \vv_\delta   - {\rm div} \, (\SSS(c_\delta,\DD\vv_\delta)) \big) \cdot \ffpsi \dx\dt = \int_{Q_{T'}} \fff \cdot \ffpsi \dx \dt \quad \text{for all }\ffpsi \in C_0^\infty(I';\mathcal{V}).
\end{equation*}
On the other hand, it is straightforward to show that the regularized concentration $c_\delta$ satisfies
\begin{equation*}
    \|\pa_t c_\delta\|_{L^2(Q_{T'})} + \|\, {\rm div} \, (c_\delta \vv_\delta)\|_{L^2(Q_{T'})} + \|\Delta c_\delta\|_{L^2(Q_{T'})} \le C,
\end{equation*}
and
\begin{equation*}
    \int_{Q_{T'}} \big(\pa_t c_\delta  + {\rm div}\,  (c_\delta\vv_\delta)  - \Delta c_\delta  \big)\, \phi\dx\dt = 0 \quad \text{for all }\phi \in C_0^\infty(I';C^\infty(\Omega)).
\end{equation*}

What remains to show is the identification of the initial conditions and the passages to the limit as $\delta\to 0$, thereby establishing the existence of a strong solution $(\vv,c)$.
This can be demonstrated precisely in the same manner as in the three-dimensional case. Furthermore, as before, the pressure term $\pi$ satisfying $\nabla \pi \in L^{\min\{p^-,\frac{2q_0}{2+q_0}\}}(\Omega)$ can be recovered by applying de Rham's theorem, thereby completing the proof of the existence of a strong solution to the system \eqref{main_sys}-\eqref{concentration}.

\subsection{Proof of Theorem \ref{uniqueness theorem2d}}
In this section, we follow the procedure detailed in Section~\ref{sec:unique} to prove uniqueness. Let $(\vv_1,c_1)$ and $(\vv_2,c_2)$ be two solutions of the problem \eqref{main_sys}-\eqref{concentration} in a two-dimensional domain. First, we take the difference between the equations \eqref{main_sys} corresponding to $\vv_1$ and $\vv_2$ and test $\vv_1 - \vv_2$. Then we have
\begin{equation} \label{uniPQ2d}
    \frac{1}{2} \frac{\rm d}{\dt} \|\vv_1 - \vv_2\|_2^2 + {\rm P} = {\rm Q}.
\end{equation}
Here, we shall employ the same notation as before in \eqref{P and Q} and \eqref{uniN}, for ${\rm P} = {\rm P_1} + {\rm P_2}$ and $Q$.
First, due to the monotonicity \eqref{P2}, we have
\begin{equation} \label{P12d}
{\rm P_1} \ge C \int_\Omega (1+ |\DD\vv_1|^2 + |\DD\vv_2|^2)^\frac{p(c_1)-2}{2}|\DD\vv_1 - \DD\vv_2|^2 \dx.
\end{equation}
Secondly, for some small $\alpha>0$, and for some $z$ between $c_1$ and $c_2$, we see that
\begin{align*} 
    |{\rm P_2}| 
    &\le C\int_\Omega (1+ |\DD\vv_1|^2 + |\DD\vv_2|^2)^{\frac{2p(z)-p(c_1)}{4} + \alpha } (1+|\DD\vv_1|^2+|\DD\vv_2|^2)^\frac{p(c_1)-2}{4}|\DD\vv_1 - \DD\vv_2||c_1 - c_2| \dx\\
    &\le C \left( \int_\Omega (1+|\DD\vv_1|^2 + |\DD\vv_2|^2)^{(2+\delta_1)(\frac{p^+}{2} - \frac{p^-}{4} + \alpha)} \dx \right)^\frac{2}{2+\delta_1} \|c_1 - c_2\|_s^2  \\
    &\hspace{0.4cm}+\epsilon \int_\Omega (1+ |\DD\vv_1|^2+ |\DD\vv_2|^2)^\frac{p(c_1)-2}{2} |\DD\vv_1 - \DD\vv_2|^2 \dx,
\end{align*}
where we have used H\"older's inequality for some small $\delta_1>0$ and sufficiently large $s>2$ with $1= \frac{1}{2+\delta_1} + \frac{1}{s} + \frac{1}{2}$. Furthermore, analogously to \eqref{highregularityDv}, we obtain
\[
\|(\overline{\DD}\vv_\delta^n)^\frac{p^-}{2}\|_{r}^{2} \le C\|\nabla ((\overline{\DD}\vv_\delta^n)^\frac{p^-}{2})\|_2^2  + C\|(\overline{\DD}\vv_\delta^n)^\frac{p^-}{2}\|_{2}^{2} \le C \calI_p(c_\delta^n,\vv_\delta^n) +C\|\overline{\DD}\vv^n_\delta\|_{p^-}^{p^-},
\]
for any $r \in [1,\infty)$ and hence, we deduce $\DD\vv \in L^{p^-}(I';L^{\frac{p^-}{2} r}(\Omega)) \cap L^\infty(I';L^2(\Omega))$ by passing $n\to\infty$ and $\delta\to 0$. If we choose $r>1$ large enough, by the interpolation inequality, we get for some small $\delta_2>0$ that
\[
\DD\vv \in L^{2+p^--\delta_2}(Q_{T'}).
\]
Note that $(2+ \delta_1)(p^+-\frac{p^-}{2} + 2 \alpha) \le 2+p^- -\delta_2$ is satisfied when $p^+ < 1+ p^-$, which automatically holds in the shear-thinning regime $1<p^-\le p^+ \le2$. Therefore, by the Sobolev embedding, we obtain
\begin{equation} \label{P22d}
|{\rm P_2}| \le F(t) \|\nabla (c_1-c_2)\|_2^2 + \epsilon \int_\Omega (1+ |\DD\vv_1|^2 + |\DD\vv_2|^2)^\frac{p(c_1)-2}{2} |\DD\vv_1-\DD\vv_2|^2 \dx, 
\end{equation}
for some $F(t) \in L^1(I')$. By collecting the estimates \eqref{P12d} and \eqref{P22d}, we get
\begin{equation} \label{P2d}
C\int_\Omega (1+ |\DD\vv_1|^2 + |\DD\vv_2|^2)^\frac{p(c_1)-2}{2}|\DD\vv_1 - \DD\vv_2|^2 \dx \le F(t) \|\nabla(c_1-c_2)\|_2^2 + {\rm P}.
\end{equation}
Furthermore, since $p^->1$, there exists $\ell\in (\frac{4}{3},2)$ such that $\frac{2-p(c_1)}{2} \cdot \frac{2\ell}{2-\ell} \le 2.$ Consequently, from the fact that $\vv_1, \vv_2 \in L^\infty(I'; W^{1,2}(\Omega))$, we have
\[
\Big\|(\overline{\DD}\vv_1)^\frac{2-p(c_1)}{2} + (\overline{\DD}\vv_2)^\frac{2-p(c_1)}{2} \Big\|_{\frac{2\ell}{2-\ell}} \in L^\infty(I').
\]
Therefore, from Lemma \ref{difference_q}, we conclude that
\begin{equation*}
    C\|\DD\vv_1 - \DD\vv_2\|_\ell^2 \le F(t) \|\nabla(c_1-c_2)\|_2^2 + {\rm P}.
\end{equation*}
For the second term ${\rm Q}$, by Hölder's inequality, we have
\begin{equation*}
    {\rm Q} = - \int_\Omega ((\vv_1-\vv_2)\cdot \nabla ) \vv_1 \cdot(\vv_1 - \vv_2)\dx
    \le \|\vv_1 - \vv_2\|_4^2 \|\nabla \vv_1\|_2 \le C \|\vv_1 - \vv_2\|_4^2,
\end{equation*}
where we have used the regularity $\vv_1 \in L^\infty(I';W^{1,2}(\Omega))$. Furthermore, since $2 < 4 < \frac{2\ell}{2-\ell}$ holds for $\ell > \frac{4}{3}$, from the Gagliardo–Nirenberg inequality, Korn's inequality and Young's inequality, we have that
\begin{equation} \label{Q2d}
    {\rm Q} \le C \|\vv_1 - \vv_2\|_2^{2\theta} \|\DD\vv_1 - \DD\vv_2 \|_\ell^{2(1-\theta)}
    \le C\|\vv_1 - \vv_2\|_2^2 + \epsilon\|\DD\vv_1 - \DD\vv_2\|_\ell^2,
\end{equation}
where $\frac{1}{4}= \frac{\theta}{2} + \frac{2-\ell}{2\ell}(1-\theta)$.
Consequently, \eqref{uniPQ2d} combined with \eqref{P2d} and \eqref{Q2d} implies that
\[
\frac{1}{2} \frac{\rm d}{\dt} \|\vv_1 - \vv_2\|_2^2 + C \|\DD\vv_1 - \DD\vv_2\|_\ell^2 \le F(t) \|\nabla (c_1 - c_2)\|_2^2 + C\|\vv_1 - \vv_2 \|_2^2.
\]

Next, taking the difference between the equations \eqref{concentration} for $c_1$ and $c_2$, we obtain
\begin{equation*}
    \pa_t(c_1-c_2) - \Delta(c_1-c_2) + {\rm div}\, (c_1\vv_1 - c_2 \vv_2) = 0.
\end{equation*}
We then test $-\Delta (c_1-c_2)$ to the above equation and use Young's inequality to get
\begin{equation*}
    \begin{aligned}
    \frac{1}{2}\frac{\rm d}{\dt} &\int_\Omega |\nabla (c_1-c_2)|^2 \dx + \int_\Omega |\nabla^2 (c_1-c_2)|^2 \dx \\
    &\le C \int_\Omega |(\vv_1 - \vv_2) \cdot \nabla c_1|^2\dx + C\int_{\Omega}|\vv_2 \cdot\nabla(c_1-c_2)|^2 dx =: {\rm U_1} + {\rm U_2}.
    \end{aligned}
\end{equation*}
\begin{comment}
To estimate $\rm U_1$, for some small $\e_1>0$, by applying H\"older's inequality, we have
\begin{equation} \label{U1}
{\rm U_1} \le \|\vv_1 -\vv_2\|_{2+\e_1}^2 \|\nabla c_1\|_a^2,
\end{equation}
where $a = (\frac{2+\e_1}{2})' = \frac{2+\e_1}{\e_1}$. Then by the Gagliardo--Nirenberg inequality, we have
\[
\|\vv_1 - \vv_2\|_{2+\e_1}^2 \le C \|\vv_1 -\vv_2\|_2^{2\theta} \|\DD\vv_1 - \DD\vv_2\|_\ell^{2(1-\theta)},
\]
where $\ell =\frac{4}{3}+$ and $\frac{1}{2+ \e_1} = \frac{\theta}{2} + \frac{1-\theta}{\ell^*}$ with $\ell^* = \frac{2\ell}{2-\ell} = 4+$. Let $\ell^*=4+ \e_2$, where $\e_2>0$ small enough. Then $\theta = \frac{2(2+\e_2-\e_1)}{(2+\e_1)(2+\e_2)}$. Also, by the Gagliardo--Nirenberg inequality again, we have
\[
\|\nabla c_1\|_a^2 \le C \|\nabla^2c_1\|_2^{2\alpha} \|\nabla c_1\|_2^{2(1-\alpha)},
\]
where $\frac{1}{a} = \frac{1-\alpha}{2}$ which means $\alpha = \frac{2-\e_1}{2+\e_1}$.
Now since $\nabla c_1 \in L^\infty(I';L^2(\Omega))$, we write \eqref{U1} as
\[
{\rm U_1} \le C \|\nabla^2c_1 \|_2^{2\alpha}\|\vv_1-\vv_2\|_2^{2\theta} \|\DD\vv_1 - \DD\vv_2\|_\ell^{2(1-\theta)}
\]
Then by Young's inequality
\[
{\rm U_1} \le C \|\nabla^2c_1\|_2^{\frac{2\alpha}{\theta}}\|\vv_1-\vv_2\|_2^2 + \epsilon\|\DD\vv_1 - \DD\vv_2\|_\ell^2.
\]
We need $\|\nabla^2c\|_2^{\frac{2\alpha}{\theta}} \in L^1(I')$, which means $\frac{2\alpha}{\theta} \le 2$, i.e., $\alpha \le \theta$. This equivalents to $0 \le \e_2$, which is always true.
\end{comment}
To estimate the first term ${\rm U_1}$, as we did in the proof of the three-dimensional case, we first note that $\vv \cdot \nabla c \in L^\infty(I';L^2(\Omega))$ for any constructed solution pair $(\vv,c)$.  Therefore, it follows that $\nabla c \in L^\nu(Q_{T'})$ for any $1 \le \nu <\infty$ by applying the maximal regularity theory with the assumption $c_0 \in W^{2,2}(\Omega)$. Now, by H\"older's inequality, we see that
\begin{equation*} 
{\rm U_1} \le \|\vv_1 - \vv_2\|_4^2 \|\nabla c_1\|_4^2 .
\end{equation*}
Furthermore, we shall use the Gagliardo--Nirenberg inequality along with the condition $2 < 4 < \frac{2\ell}{2-\ell}$, Korn's inequality and Young's inequality to obtain
\begin{align*}
{\rm U_1} &\le C\|\nabla c_1\|_4^2 \|\vv_1 - \vv_2\|_2^{2\theta}\|\DD\vv_1 - \DD\vv_2\|_\ell^{2(1-\theta)} \\
&\le C \|\nabla c_1\|_4^\frac{2}{\theta} \|\vv_1-\vv_2\|_2^2 + \epsilon\|\DD\vv_1-\DD\vv_2\|_\ell^2, 
\end{align*}
where $\frac{1}{4} = \frac{\theta}{2} + \frac{2-\ell}{2\ell}(1-\theta)$. Thus, using the fact that $\nabla c_1 \in L^\nu(Q_{T'})$ for any $1 \le \nu < \infty$, we conclude that
\[
{\rm U_1} \le G(t)\|\vv_1 - \vv_2\|_2^2 + \epsilon \|\DD\vv_1 - \DD\vv_2\|_\ell^2,
\]
for some $G(t) \in L^1(I')$. For the second term $\rm U_2$, by H\"older inequality and the Sobolev embedding, we see that
\[
{\rm U_2} \le \|\vv_2\|_{12}^2 \|\nabla c_1 - \nabla c_2\|_\frac{12}{5}^2 \le \|\nabla \vv_2\|_2^2 \|\nabla c_1- \nabla c_2\|_\frac{12}{5}^2 \le C \|\nabla c_1- \nabla c_2\|_\frac{12}{5}^2,
\]
where we have used $\vv \in L^\infty(I';W^{1,2}(\Omega))$ for the last inequality. Therefore, applying the Gagliardo--Nirenberg inequality and Young's inequality yields
\[
{\rm U_2} \le C\|\nabla(c_1- c_2)\|_2^2 + \epsilon \|\nabla^2(c_1-c_2)\|_2^2.
\]
Collecting all the estimates, we have
\begin{align*}
&\frac{\rm d}{\dt} \|\vv_1 - \vv_2\|_2^2 + \frac{\rm d}{\dt} \|\nabla(c_1-c_2)\|_2^2 +C \|\nabla^2(c_1-c_2)\|_2^2 + C \|\DD\vv_1 - \DD\vv_2\|_\ell^2 \\
&\le \widetilde{F}(t) \|\nabla (c_1 - c_2)\|_2^2 + \widetilde{G}(t)\|\vv_1 - \vv_2 \|_2^2,
\end{align*}
for some $\widetilde{F}(t), \widetilde{G}(t) \in L^1(I')$.
Therefore, we apply Gr\"onwall's inequality to obtain
\[
\vv_1 - \vv_2 = 0, \quad \nabla(c_1 - c_2) = 0,
\]
and hence, the application of Poincar\'e's inequality completes the proof.

\section{Conclusion}\label{sec:conc}

In this paper, we have established the existence of a unique local-in-time strong solution for a system describing generalized Newtonian fluids with a concentration-dependent power-law index in the shear-thinning regime, coupled with a convection-diffusion equation in three spatial dimensions, under the assumption of spatial periodicity. The key distinction from previous works, including~\cite{Diening2005} and \cite{choi2024}, lies in the fact that differentiating the Cauchy stress tensor yields additional terms involving \( \nabla c \) and \( \partial_t c \), which require further technical treatment to balance the associated exponents. Moreover, since the Gr\"onwall inequality must be applied to the approximate concentration at the continuous level, it is necessary to ensure appropriate temporal regularity. To impose minimal assumptions on the initial concentration  $c_0$ , a regularization technique was introduced. In addition, by further assuming $p^+ < \frac{2p^- - 2}{2 - p^-} p^-$, we also established the uniqueness of the strong solution. The proof presented herein can be readily extended to the two-dimensional case, in which the same result holds under a less restrictive condition on the exponent.

An interesting direction for future research is to establish similar results in the whole domain $\R^3$, and one may also consider extending the current local solution result to a result over an arbitrary time interval. Furthermore, as the existence and uniqueness have been proven in a physically meaningful setting, it becomes relevant from a practical viewpoint to design finite element approximations of such solutions, along with conducting a convergence analysis and deriving optimal error estimates. These topics will be addressed in a forthcoming works.

\section*{Acknowledgements}
K. Kang is supported by the National Research Foundation of Korea Grant funded by the Korea Government (RS-2024-00336346 and RS-2024-00406821).
Seungchan Ko is supported by National Research Foundation of
Korea Grant funded by the Korean Government (RS-2023-00212227). Kyueon Choi is supported by the National Research Foundation of Korea Grant funded by the Korea Government (RS-2023-00212227 and RS-2024-00406821).

\bibliography{references}
\bibliographystyle{abbrv}

%\section*{Appendix. Proof of Theorem \ref{thm:simple_system} and \ref{thm:geometric_length}}

\end{document}